\documentclass[a4paper]{article}

\usepackage{ifthen}
\usepackage{amsmath}
\usepackage{amssymb}
\usepackage{amsthm}
\usepackage{bm}
\usepackage{enumerate}
\usepackage{textcomp}
\usepackage{pifont}
\usepackage{stmaryrd}
\usepackage{wasysym}
\usepackage[dvips]{graphics}
\usepackage{rotating}
\usepackage{cancel}

\newfont{\ffi}{cmfi10 scaled 1000}

\newcommand{\mc}[1]{{\mathcal{#1}}}
\newcommand{\mf}[1]{{\mathfrak{#1}}}
\newcommand{\bb}[1]{{\mathbb{#1}}}

\makeatletter
\newcommand\appendixsection{\@startsection {section}{1}{\z@}
	{-3.5ex \@plus -1ex \@minus -.2ex}{2.3ex \@plus.2ex}
	{\normalfont\Large\bfseries\hspace*{-15pt}Appendix\ }}
\makeatother

\DeclareMathOperator{\dom}{dom}
\DeclareMathOperator{\sign}{sign}
\DeclareMathOperator{\spn}{span}
\DeclareMathOperator{\supp}{supp}
\DeclareMathOperator{\clos}{clos}
\DeclareMathOperator{\Sub}{Sub}
\DeclareMathOperator{\Adm}{Adm}
\newcommand{\PW}{{\mc P\hspace*{-1pt}W\!}}

\DeclareMathOperator{\RE}{Re}
\DeclareMathOperator{\IM}{Im}
\renewcommand{\Re}{\RE}
\renewcommand{\Im}{\IM}

\newcommand{\qu}{\overline}
\newcommand{\HB}{\mc H\!B}
\DeclareMathOperator{\Assoc}{Assoc}
\DeclareMathOperator{\mt}{mt}

\newcommand{\comment}[1]{}

\newlength{\maxlabwidth}
\newenvironment{axioms}[1]{
    \setlength{\maxlabwidth}{#1}
    \begin{list}{}{
    \setlength{\rightmargin}{2mm}
    \setlength{\leftmargin}{\maxlabwidth}\addtolength{\leftmargin}{2mm}
    \setlength{\labelsep}{0mm}
    \setlength{\labelwidth}{\maxlabwidth}
    \setlength{\itemindent}{0mm}
    
    }
    }{
    \end{list}
    }

\numberwithin{equation}{section}

\swapnumbers
\theoremstyle{plain}
	\newtheorem{lem}{Lemma}[section]
	\newtheorem{pro}[lem]{Proposition}
	\newtheorem{thm}[lem]{Theorem}
	\newtheorem{cor}[lem]{Corollary}
	\newtheorem{namth}[lem]{}
\theoremstyle{definition}
	\newtheorem{defi}[lem]{Definition}
\theoremstyle{remark}
	\newtheorem{rem}[lem]{Remark}
	\newtheorem{exa}[lem]{Example}
	\newtheorem{namre}[lem]{}
\renewcommand{\qedsymbol}{\raisebox{-2pt}{\large\ding{113}}}
\newcommand{\defendsymbol}{}
\newcommand{\qedsymbolsave}{\qedsymbol}
\newenvironment{lemma}[2][]{\begin{lem}[#1]\label{LE#2}}{\end{lem}}
\newcommand{\leref}[1]{Lemma \ref{LE#1}}
\newenvironment{proposition}[2][]{\begin{pro}[#1]\label{PR#2}}{\end{pro}}
\newcommand{\prref}[1]{Proposition \ref{PR#1}}
\newenvironment{theorem}[2][]{\begin{thm}[#1]\label{TH#2}}{\end{thm}}
\newcommand{\thref}[1]{Theorem \ref{TH#1}}
\newenvironment{corollary}[2][]{\begin{cor}[#1]\label{CO#2}}{\end{cor}}
\newcommand{\coref}[1]{Corollary \ref{CO#1}}
\newenvironment{ntheorem}[2][]{\begin{namth}\textbf{\!{#1:}\!}\label{NT#2}}{\end{namth}}

\newenvironment{definition}[2][]{\begin{defi}[#1]\label{DE#2}}{
	\renewcommand{\qedsymbolsave}{\qedsymbol}\renewcommand{\qedsymbol}{\defendsymbol}
	\popQED{\qed}\renewcommand{\qedsymbol}{\qedsymbolsave}\end{defi}}
\newcommand{\deref}[1]{Definition \ref{DE#1}}
\newenvironment{remark}[2][]{\begin{rem}[#1]\label{RE#2}}{
	\renewcommand{\qedsymbolsave}{\qedsymbol}\renewcommand{\qedsymbol}{\defendsymbol}
	\popQED{\qed}\renewcommand{\qedsymbol}{\qedsymbolsave}\end{rem}}
\newcommand{\reref}[1]{Remark \ref{RE#1}}
\newenvironment{example}[2][]{\begin{exa}[#1]\label{EX#2}}{
	\renewcommand{\qedsymbolsave}{\qedsymbol}\renewcommand{\qedsymbol}{\defendsymbol}
	\popQED{\qed}\renewcommand{\qedsymbol}{\qedsymbolsave}\end{exa}}
\newcommand{\exref}[1]{Example \ref{EX#1}}

\newenvironment{proofof}[1]{\begin{proof}[\textit{Proof (of #1)}]}{\end{proof}}
\newcommand{\bibi}[5]{\bibitem[#5]{#1} \textsc{#2}:\ \textit{#3,}\ {#4.}}
\begin{document}

{\Large\bf
\begin{flushleft}
	Majorization in de~Branges spaces I. Representability of subspaces
\end{flushleft}
}
\vspace*{3mm}
\begin{center}
	{\sc Anton Baranov, Harald Woracek}
\end{center}


\begin{abstract}
	\noindent
	In this series of papers we study subspaces of de~Branges spaces of entire functions 
	which are generated by majorization on subsets $D$ of the closed upper half-plane. The present, 
	first, part is addressed to the question which subspaces of a given de~Branges space can be represented 
	by means of majorization. Results depend on the set $D$ where majorization is permitted. 
	Significantly different situations are encountered when $D$ is close to the real axis or 
	accumulates to $i\infty$. 
\end{abstract}
\begin{flushleft}
	{\small
	{\bf AMS Classification Numbers:} 46E20, 30D15, 46E22 
	}
\end{flushleft}


%
%
\section{Introduction}

In the paper \cite{debranges:1959} L. de~Branges initiated the study of Hilbert spaces of entire 
functions, which satisfy specific additional axioms. These spaces can be viewed as a generalization of the 
classical Paley--Wiener spaces $\PW_a$, which consist of all entire functions of exponential type at most $a$ whose 
restriction to the real line is square-integrable. The theory of de~Branges spaces can be viewed as a 
generalization of classical Fourier analysis. For example, their structure theory gives rise 
to generalizations of the Paley--Wiener Theorem, which identifies $\PW_a$ as the Fourier image of all 
square-integrable functions supported in the interval $[-a,a]$. De~Branges spaces also appear in many other 
areas of analysis, like the theory of Volterra operators and entire operators in the sense of M.G. Kre\u{\i}n, 
the shift operator in the Hardy space, V.P. Potapov's $J$-theory, 
the spectral theory of Schr\"odinger operators, Stieltjes or Hamburger power moment problems, or prediction theory of 
Gaussian processes, cf.\ \cite{gohberg.krein:1970}, \cite{gorbachuk.gorbachuk:1997}, \cite{nikolski:1986}, 
\cite{golinskii.mikhailova:1997}, \cite{remling:2002}, \cite{dym.mckean:1976}. 

The present paper is the first part of a series, in which we investigate the aspect of 
majorization in de~Branges spaces. Such considerations have a 
long history in complex analysis, going back to the Beurling--Malliavin Multiplier Theorem, cf.\ \cite{beurling.malliavin:1962}. 
In recent investigations by V. Havin and J. Mashreghi, results of this kind were proven 
in the more general setting of shift-coinvariant 
subspaces of the Hardy space, cf.\ \cite{havin.mashreghi:2003a}, \cite{havin.mashreghi:2003b}. 
All these considerations, as well as our previous work \cite{baranov.woracek:dbmaj}, deal with majorization along the 
real line. 

Having these concepts in mind, a general notion of majorization in de~Branges spaces evolves: 

\begin{definition}{A1}
	Let $\mc H$ be a de~Branges space, and let $\mf m:D\to[0,\infty)$ where 
	$D\subseteq\bb C^+\cup\bb R$. Set 
	\[
		R_{\mf m}(\mc H):= \big\{\,F\in\mc H:\ \exists\,C>0:|F(z)|,|F^\#(z)|\leq C\mf m(z),
		\ z\in D\,\big\}
		\,,
	\]
	and define
	\[
		\mc R_{\mf m}(\mc H):= \clos_{\mc H} R_{\mf m}(\mc H)
		\,.
	\]
\end{definition}

\noindent
It turns out that, provided $\mc R_{\mf m}(\mc H)\neq\{0\}$ and $\mf m$ satisfies a mild regularity condition, the space 
$\mc R_{\mf m}(\mc H)$ is a de~Branges subspace of $\mc H$, i.e.\ is itself a de~Branges space when endowed with the 
inner product inherited from $\mc H$. 

The following questions related to this concept come up naturally. 

\hspace*{0pt}\\[-2mm]* {\it Which de~Branges subspaces $\mc L$ of a given de~Branges space $\mc H$ can be realized as 
	$\mc L=\mc R_{\mf m}(\mc H)$ with some majorant $\mf m$ ?}

\hspace*{0pt}\\[-2mm]* {\it If $\mc L$ is of the form $\mc R_{\mf m}(\mc H)$ with some $\mf m$, 
	how big or how small can $\mf m$ be chosen such that still $\mc L=\mc R_{\mf m}(\mc H)$ ?} 
\hspace*{0pt}\\[-2mm] 

Let us point out the two aspects of the second question. If $\mc L=\mc R_{\mf m}(\mc H)$, we have 
available a dense linear subspace of $\mc L$ which consists of functions with limited growth on the domain $D$ of $\mf m$, namely 
$R_{\mf m}(\mc H)$. This knowledge becomes stronger, the smaller $\mf m$ is. 
On the other hand, the equality $\mc L=\mc R_{\mf m}(\mc H)$ 
also says that an element of $\mc H$ already belongs to $\mc L$ if it is majorized by $\mf m$. This knowledge becomes 
stronger, the bigger $\mf m$ is. 

Answers to these questions will, of course, depend on the set $D$ where majorization is permitted. Up to now, only majorization 
along $\bb R$ has been considered. For this case, the first question has been answered completely in 
\cite{baranov.woracek:dbmaj}. The "how small"--part of the second question is related to the deep investigations in 
\cite{havin.mashreghi:2003a}, \cite{havin.mashreghi:2003b}. 

In this paper we give some answers to the first question, and to the "how big"--part of the second question. As domains $D$ 
of majorization we consider, among others, rays contained in the closed upper half-plane, lines parallel to the real 
axis contained in the closed upper half-plane, or combinations of such types of sets. For example, it turns out that 
each de~Branges subspace $\mc L$ of any given de~Branges space $\mc H$ can be realized as $\mc R_{\mf m}(\mc H)$, when 
majorization is allowed on $\bb R\cup i[0,\infty)$. Even more, one can choose for $\mf m$ a majorant which is naturally associated 
to $\mc L$, does not depend on the external space $\mc H$, is quite big, and actually gives $\mc L=R_{\mf m}(\mc H)$. 
It is an interesting and, on first sight, maybe surprising consequence 
of de~Branges' theory, that the main strength of majorization 
is contributed by boundedness along the imaginary half-line, and not 
along $\bb R$. In fact, if we permit majorization only on some 
ray $i[h,\infty)$ where $h>0$, then all de~Branges subspaces subject 
to an obvious necessary condition can be realized in the way stated above. Similar phenomena, where growth restrictions 
on the imaginary half-axis imply a certain behaviour along the real line, have already been experienced in the 
classical theory, see e.g.\ \cite[Theorem 2]{baranov:2001} or \cite[Theorem 26]{debranges:1968}. 

Let us close this introduction with an outline of the organization of this paper. 
In order to make the presentation as self-contained as possible, 
we start in Section 2 with recalling some basic definitions and collecting some results which are essential for 
what follows, among them, the definition of de~Branges spaces of entire functions, their relation to 
entire functions of Hermite--Biehler class, and the structure of de~Branges subspaces. 
In Section 3, we make precise under which conditions on $\mf m$ the space $\mc R_{\mf m}(\mc H)$ becomes a de~Branges 
subspace of $\mc H$, and discuss some examples of majorants. Sections 4 and 5 contain the main results of this paper. First 
we deal with representation of de~Branges subspaces by majorization along rays not parallel to the real axis. 
Then we turn to spaces $\mc R_{\mf m}(\mc H)$ obtained when majorization is required on a set close to the real axis, for 
example a line parallel to $\bb R$. The paper closes with two appendices. In the first appendix we prove 
an auxiliary result on model subspaces generated by inner functions, which is employed in Section 5. We decided to move this 
theorem out of the main text, since it is interesting on its own right and independent of the presentation 
concerning de~Branges spaces. In the second appendix, we are summing up the representation theorems 
for de~Branges subspaces obtained in Sections 4 and 5 in tabularic form. 
\medskip

\noindent
{\bf Acknowledgements.} We would like to thank Alexei Poltoratski who 
suggested the use of weak type estimates for the proof of \thref{A60}. 
The first author was partially supported by the grants
MK-5027.2008.1 and NSH-2409.2008.1.

\section{Preliminaries}
%
%

%
%
\begin{flushleft}
	\textbf{I. Mean type and zero divisors}
\end{flushleft}
\vspace*{-2mm}
We will use the standard theory of Hardy spaces in the half-plane as presented e.g.\ 
in \cite{garnett:1981} or \cite{rosenblum.rovnyak:1994}. In this place, let us only recall the following notations. 
We denote by 
\begin{enumerate}[$(i)$]
\item $\mc N=\mc N(\bb C^+)$ the set of all functions of \emph{bounded type}, that is, 
	of all functions $f$ analytic in $\bb C^+$, which can be represented as a quotient $f=g^{-1}h$ of two 
	bounded and analytic functions $g$ and $h$. 
\item $\mc N_+=\mc N_+(\bb C^+)$ the \emph{Smirnov class}, that is, 
	the set of all functions $f$ analytic 
	in $\bb C^+$, which can be represented as $f=g^{-1}h$ with two bounded and analytic functions $g$ and $h$ where 
	in addition $g$ is \emph{outer}. 
\item $H^2=H^2(\bb C^+)$ the \emph{Hardy space}, that is, the set 
	of all functions $f$ analytic in $\bb C^+$ 
	which satisfy 
	\[
		\sup_{y>0}\int_{\bb R}|f(x+iy)|^2\,dx<\infty
		\,.
	\]
\end{enumerate}
If $f\in\mc N$, the \emph{mean type of $f$} is defined by the formula 
\[
	\mt f:= \limsup_{y\to+\infty}\frac 1y\log|f(iy)|
	\,.
\]
Then $\mt f\in\bb R$, and the radial growth of $f$ is determined by the number $\mt f$ in the 
following sense: For every $a\in\bb R$ and $0<\alpha<\beta<\pi$, 
there exists an open set $\Delta_{a,\alpha,\beta}\subseteq(0,\infty)$ with finite logarithmic length, such that 
\begin{equation}\label{A6}
	\lim_{\substack{r\to\infty\\ r\not\in\Delta_{a,\alpha,\beta}}} \frac 1r \log\big|f(a+re^{i\theta})\big|=
	\mt f\cdot\sin\theta
	\,,
\end{equation}
uniformly for $\theta\in[\alpha,\beta]$. If, for some $\epsilon>0$, the angle $[\alpha-\epsilon,\beta+\epsilon]$ 
does not contain any zeros of $f(a+z)$, then one can choose $\Delta_{a,\alpha,\beta}=\emptyset$. 

Here we understand by the logarithmic length of a subset $M$ of $\bb R^+$ the value of the
integral $\int_M x^{-1}\,dx$. When speaking about logarithmic length of a set $M$, we always
include that $M$ should be measurable.

\begin{definition}{A7}
	Let $\mf m:D\to\bb C$ be a function defined on some subset $D$ of the complex plane. 
	\begin{enumerate}[$(i)$]
	\item By analogy with \eqref{A6} we define the \emph{mean type of $\mf m$} as 
		\[
			\mt_{\mc H}\mf m:= \inf\Big\{\frac 1{\sin\theta} 
			\limsup_{\substack{r\to\infty\\ r\in M}}
			\frac 1r\log |\mf m(a+re^{i\theta})| \Big\}\in[-\infty,+\infty]
			\,,
		\]
		where the infimum is taken over those values 
		$a\in\bb R$, $\theta\in(0,\pi)$, and those sets $M\subseteq\bb R^+$ of infinite logarithmic length, 
		for which $\{a+re^{i\theta}:\,r\in M\}\subseteq D$. Thereby we understand the infimum of the empty set 
		as $+\infty$. 
	\item We associate to $\mf m$ its \emph{zero divisor} $\mf d_{\mf m}:\bb C\to\bb N_0\cup\{\infty\}$. If $w\in\bb C$, 
		then $\mf d_{\mf m}(w)$ is defined as the infimum of all numbers $n\in\bb N_0$, such that there 
		exists a neighbourhood $U$ of $w$ with the property 
		\[
			\inf_{\substack{z\in U\cap D\\ |z-w|^n\neq 0}}\frac{|\mf m(z)|}{|z-w|^n}>0
			\,.
		\]
	\end{enumerate}
\end{definition}

\noindent
Note that in general $\mt\mf m$ may take the values $\pm\infty$. However, the above definition ensures that $\mt\mf m$ 
coincides with the classical notion in case $\mf m\in\mc N$. 

A similar remark applies to $\mf d_{\mf m}$. If $D$ is open, and $\mf m$ is analytic, then $\mf d_{\mf m}|_D$ is just 
the usual zero divisor of $\mf m$, i.e.\ $\mf d_{\mf m}(w)$ is the multiplicity of the point $w$ as a zero of $\mf m$ whenever 
$w\in D$. Moreover, note that the definition of $\mf d_{\mf m}$ is made in such a way 
that $\mf d_{\mf m}(w)=0$ whenever $w\not\in\qu D$.

%
%
\begin{flushleft}
	\textbf{II. Axiomatics of de~Branges spaces of entire functions}
\end{flushleft}
\vspace*{-2mm}
Our standard reference concerning the theory of de~Branges spaces of entire functions is \cite{debranges:1968}. In this and 
the following two subsections we will recall some basic facts about de~Branges spaces. Our aim is not only 
to set up the necessary notation, but also to put emphasis on those results which are significant in the context of the 
present paper. 

We start with the axiomatic definition of a de~Branges space. 

\begin{definition}{A14}
	A \emph{de~Branges space} is a Hilbert space $\langle\mc H,(\cdot,\cdot)\rangle$, $\mc H\neq\{0\}$, with
	the following properties:
	\begin{axioms}{14mm}
	\item[dB1] The elements of $\mc H$ are entire functions, and
		for each $w\in\bb C$ the point evaluation $F\mapsto F(w)$
		is a continuous linear functional on $\mc H$. 
	\item[dB2] If $F\in\mc H$, also $F^\#(z):= \overline{F(\bar z)}$
		belongs to $\mc H$ and $\|F^\#\|=\|F\|$.
	\item[dB3] If $w\in\bb C\setminus\bb R$ and $F\in\mc H$,
		$F(w)=0$, then
		\[
			\frac{z-\bar w}{z-w}F(z)\in\mc H\quad \text{ and }\quad \Big\|\frac{z-\bar w}{z-w}F(z)\Big\|
			=\big\|F\big\|
			\,.
		\]
	\end{axioms}
\end{definition}

By (dB1) a de~Branges space $\mc H$ is a reproducing kernel Hilbert space. We will denote 
the kernel corresponding to $w\in\bb C$ by $K(w,\cdot)$ or, if it is necessary to be more specific, by $K_{\mc H}(w,\cdot)$. 
A particular role is played by the norm of reproducing kernel functions. We will denote 
\[
	\nabla_{\!\mc H}(z):= \|K(z,\cdot)\|_{\mc H}, \qquad z\in\bb C
	\,. 
\]
This norm can be computed e.g.\ as 
\[
	\nabla_{\!\mc H}(z)=\sup\big\{|F(z)|:\,\|F\|_{\mc H}=1\big\}=\big(K(z,z)\big)^{1/2}
	\,. 
\]
Let us explicitly point out that every element of $\mc H$ is majorized by $\nabla_{\!\mc H}$: By the 
Schwarz inequality we have 
\begin{equation}\label{A21}
	|F(z)|\leq\|F\|\nabla_{\!\mc H}(z),\qquad z\in\bb C, \ \  F\in\mc H
	\,. 
\end{equation}

\begin{remark}{A29}
	Let $\mc H$ be a de~Branges space. For a subset $L\subseteq\mc H$ we define $\mf d_L:\bb C\to\bb N_0$ 
	as 
	\[
		\mf d_L (w):= \min_{F\in L}\mf d_F(w)
		\,.
	\]
	Due to the axiom (dB3), we have $\mf d_{\mc H}(w)=0$, $w\in\bb C\setminus\bb R$. 
	In fact, if $F\in\mc H$ and $w$ is a nonreal zero of $F$, then $(z-w)^{-1}F(z)\in\mc H$. 
	This need not be true for real points $w$. However, one can show that, if $w\in\bb R$ and 
	$\mf d_F(w)>\mf d_{\mc H}(w)$, then $(z-w)^{-1}F(z)\in\mc H$. 
\end{remark}

\begin{remark}{A30}
	Let $\mc H$ be a de Branges space, and let $\mf m:D\to\bb C$ be a function defined on 
	some subset $D$ of the complex plane. We define the \emph{mean type of $\mf m$ relative to $\mc H$} by 
	\[
		\mt_{\mc H}\mf m:=  \mt \frac{\mf m}{\nabla_{\!\mc H}}.
	\]
	If $L$ is a subset of $\mc H$, the \emph{mean type of $L$ relative to $\mc H$} is 
	\[
		\mt_{\mc H}L:= \sup_{F\in L}\mt_{\mc H} F
		\,.
	\]
	Note that, by \eqref{A21}, we have $\mt_{\mc H}L\leq 0$. 
	
	For each $\alpha\leq 0$ the set $\{F\in\mc H:\,\mt_{\mc H} F \leq\alpha\}$ is closed, cf.\ 
	\cite{kaltenbaeck.woracek:growth}. This implies that always $\mt_{\mc H}\clos_{\mc H}L=\mt_{\mc H}L$. 
\end{remark}

\begin{remark}{A31}
	For a de~Branges space $\mc H$ let $S_{\mc H}$ denote the operator of multiplication by the independent 
	variable. That is, 
	\[
		(S_{\mc H}F)(z):= zF(z),\qquad \dom S_{\mc H}:= \big\{F\in\mc H:\,zF(z)\in\mc H\big\}
		\,.
	\]
	The relationship between de~Branges spaces and entire operators in the sense of M.G. Kre\u{\i}n 
	is based on the fact that $S_{\mc H}$ is a closed symmetric operator with defect index $(1,1)$ for which 
	every complex number is a point of regular type. 
\end{remark}

\begin{remark}{A24}
	Taking up the operator theoretic viewpoint, the role played by \emph{functions associated to 
	$\mc H$} can be explained neatly. For a de~Branges space $\mc H$, the set of functions associated to $\mc H$ 
	can be defined as 
	\[
		\Assoc\mc H:= \big\{G_1(z)+zG_2(z):\,G_1,G_2\in\mc H\big\}
		\,. 
	\]
	Clearly, $\Assoc\mc H$ is a linear space which contains $\mc H$. 
	
	The space $\Assoc\mc H$ can be used to describe the extensions of $S_{\mc H}$ by means 
	of difference quotients. We have 
	\begin{multline*}
		F\!\in\!\Assoc \mc H \ \Longleftrightarrow \ 
		\forall\,G\!\in\!\mc H,\ w\!\in\!\bb C:\ \frac{F(z)G(w)-F(w)G(z)}{z-w}\!\in\!\mc H. 
	\end{multline*}
	Moreover, for each $F\in\Assoc\mc H$ and $w\in\bb C$, $F(w)\neq 0$, the difference quotient operator 
	\[
		\rho_{F,w}:\ G \mapsto \frac{G(z)-\frac{G(w)}{F(w)}F(z)}{z-w}
	\]
	is a bounded linear operator of $\mc H$ into itself, actually, the resolvent of some extension of $S_{\mc H}$. 
	Let us note that, if $F$ is not only associated to $\mc H$ but belongs to $\mc H$, 
	we have $\rho_{F,w}\mc H=\dom S_{\mc H}$, $F(w)\neq 0$. 
\end{remark}

%
%
\begin{flushleft}
	\textbf{III. De~Branges spaces and Hermite--Biehler functions}
\end{flushleft}
\vspace*{-2mm}
It is a basic fact that a de~Branges space $\mc H$ is completely determined by a single entire function. 

\begin{definition}{A22}
	We say that an entire function $E$ belongs to the \emph{Hermite--Biehler class $\HB$}, if
	\[
		|E^\#(z)|<|E(z)|,\qquad z\in\bb C^+
		\,.
	\]
	If $E\in\HB$, define
	\[
		\mc H(E):= \Big\{F\text{ entire}:\,\frac FE,\frac{F^\#}E\in H^2(\bb C^+)\Big\}
		\,,
	\]
	and 
	\[
		(F,G)_E:= \int_{\bb R}\frac{F(t)\qu{G(t)}}{|E(t)|^2}\,dt,\qquad
		F\in\mc H(E)
		\,.
	\]
\end{definition}

Instead of $E^{-1}F,E^{-1}F^\#\in H^2$ one could, equivalently, require that 
$E^{-1}F$ and $E^{-1}F^\#$ are of bounded type 
and nonpositive mean type in the upper half-plane, and that $\int_{\bb R}|E^{-1}(t)F(t)|^2\,dt<\infty$. 
This is, in fact, the original definition in \cite{debranges:1959}. 

The relation between de~Branges spaces and Hermite-Biehler functions is established by the following fact: 

\begin{ntheorem}[De~Branges spaces via $\HB$]{A28}
	For every function $E\in\HB$, the space $\langle\mc H(E),(\cdot,\cdot)_E\rangle$ is a de~Branges space, 
	and conversely every de~Branges space can be obtained in this way.
\end{ntheorem}
	
The function $E\in\HB$ which realizes a given de~Branges space $\langle\mc H,(\cdot,\cdot)\rangle$ as 
$\langle\mc H(E),(\cdot,\cdot)_E\rangle$ is not unique. 
However, if $E_1,E_2\in\HB$ and $\langle\mc H(E_1),(\cdot,\cdot)_{E_1}\rangle=\langle\mc H(E_2),(\cdot,\cdot)_{E_2}\rangle$, 
then there exists a constant $2\times 2$-matrix $M$ with real entries and determinant $1$, such that 
\[
	(A_2,B_2)=(A_1,B_1)M
	\,.
\]
Here, and later on, we use the generic decomposition of a function $E\in\HB$ as $E=A-iB$ with 
\begin{equation}\label{A16}
	A:= \frac{E+E^\#}2,\qquad B:= i\frac{E-E^\#}2
	\,.
\end{equation}
For each two function $E_1,E_2\in\HB$ with $\langle\mc H(E_1),(\cdot,\cdot)_{E_1}\rangle=
\langle\mc H(E_2),(\cdot,\cdot)_{E_2}\rangle$, there exist constants $c,C>0$ such that 
\[
	c|E_1(z)|\leq|E_2(z)|\leq C|E_1(z)|,\qquad z\in\bb C^+\cup\bb R
	\,.
\]
The notion of a phase function is important 
in the theory of de~Branges spaces. For $E\in\HB$, 
a \emph{phase function} of $E$ is a continuous, increasing function 
$\varphi_E:\bb R\to\bb R$ with $E(t)\exp (i\varphi_E(t))\in\bb R$, 
$t\in\bb R$. A phase function $\varphi_E$ is 
by this requirement defined uniquely up to an additive constant which 
belongs to $\pi\bb Z$. Its derivative is continuous, positive, and can be computed as 
\begin{equation}\label{A33}
	\varphi'(t)=\pi\frac{K(t,t)}{|E(t)|^2}=a+\sum_n\frac{|\Im z_n|}{|t-z_n|^2}
	\,,
\end{equation}
where $z_n$ are zeros of $E$ listed according to their multiplicities, and $a:= -\mt(E^{-1}E^\#)$. 

\begin{remark}{A25}
	Let $\langle\mc H,(\cdot,\cdot)\rangle$ be a de~Branges space, and let $E\in\HB$ be such that 
	$\langle\mc H,(\cdot,\cdot)\rangle=\langle\mc H(E),(\cdot,\cdot)_E\rangle$. 
	Then all information about $\mc H$ can, theoretically, be 
	extracted from $E$. In general this is a difficult task, however, for some items it can be done explicitly. For example: 
	\begin{enumerate}[$(i)$]
	\item The reproducing kernel $K(w,\cdot)$ of $\mc H$ is given as 
			\[
				K(w,z)=\frac{E(z)E^\#(\bar w)-E(\bar w)E^\#(z)}{2\pi i(\bar w-z)}
				\,.
			\]
			In particular, this implies that $E\in\Assoc\mc H$. 
	\item We have $\mf d_{\mc H}=\mf d_E$. This equality even holds if we only assume that 
		$\mc H=\mc H(E)$ as sets, i.e.,	without assuming equality of norms. 
	\item The function $\nabla_{\!\mc H}$ is given as 
		\begin{equation}\label{A23}
			\nabla_{\!\mc H}(z)=
			\begin{cases}
				\Big(\frac{|E(z)|^2-|E(\qu z)|^2}{4\pi\Im z}\Big)^{1/2}, &\hspace*{-2mm}\quad z\in\bb C\setminus\bb R,\\
				\pi^{-1/2}|E(z)|(\varphi_E'(z))^{1/2}, &\hspace*{-2mm}\quad z\in\bb R.
			\end{cases}
		\end{equation}
		In particular, we have $\mf d_{\nabla_{\!\mc H}}=\mf d_{\mc H}$. 
	\item We have 
		\[
			\mt_{\mc H} F=\mt\frac FE,\qquad F\in\mc H
			\,.
		\]
		This follows from the estimates (with $w_0\in\bb C^+$ fixed)
		\begin{equation}\label{A42}
			\frac{|E(w_0)|(1-|\frac{E(\qu{w_0})}{E(w_0)}|)}{2\pi\nabla_{\!\mc H}(w_0)}\frac 1{|z-\qu w_0|}\leq
			\frac{\nabla_{\!\mc H}(z)}{|E(z)|}\leq \frac 1{2\sqrt\pi}\frac 1{\sqrt{\Im z}},\quad z\in\bb C^+
			\,,
		\end{equation}
		which are deduced from the inequality 
                $|K(w_0,z)|=\big|\big( K(w_0,\cdot),K(z,\cdot) \big)\big|
                \leq\nabla_{\!\mc H}(w_0)\nabla_{\!\mc H}(z)$ and \eqref{A23}. 
	\end{enumerate}
\end{remark}

%
%
\begin{flushleft}
	\textbf{IV. Structure of dB-subspaces}
\end{flushleft}
\vspace*{-2mm}
The, probably, most important notion in the theory of de~Branges spaces is the one of de~Branges subspaces. 

\begin{definition}{A43}
	A subset $\mc L$ of a de~Branges space $\mc H$ is called a \emph{dB-subspace} of $\mc H$,
	if it is itself, with the norm inherited from $\mc H$, 
	a de~Branges space. 
	
	We will denote the set of all dB-subspaces of a given space $\mc H$ by $\Sub\mc H$.
	If $\mf d:\bb C\to\bb N_0$, we set 
	\[
		\Sub_{\mf d}\mc H:= \big\{\mc L\in\Sub\mc H:\,\mf d_{\mc L}=\mf d \big\}
		\,.
	\]
\end{definition}

\noindent
Since dB-subspaces with $\mf d_{\mc L}=\mf d_{\mc H}$ appear quite frequently, we introduce the shorthand notation 
$\Sub^*\mc H:= \Sub_{\mf d_{\mc H}}\mc H$. 

It is apparent from the axioms (dB1)--(dB3) of \deref{A14} that a subset 
$\mc L$ of $\mc H$ is a dB-subspace if and only if the following three conditions hold:
\begin{enumerate}[$(i)$]
\item $\mc L$ is a closed linear subspace of $\mc H$; 
\item If $F\in\mc L$, then also $F^\#\in\mc L$; 
\item If $F\in\mc L$ and $z_0\in\bb C\setminus\bb R$ is such that $F(z_0)=0$, then $\frac{F(z)}{z-z_0}\in\mc L$. 
\end{enumerate}

\begin{example}{A26}
	Some examples of dB-subspaces can be obtained by imposing conditions on real zeros or on mean type. 
	
	If $\mf d:\bb C\to\bb N_0$, $\supp\mf d\subseteq\bb R$, 
	is a function such that $\mf d_{F_0}\geq\mf d$ for some $F_0\in\mc H\setminus\{0\}$, then 
	\[
		\mc H_{\mf d}:= \big\{F\in\mc H:\,\mf d_F\geq\mf d\big\}\in\Sub\mc H
		\,.
	\]
	We have $\mf d_{\mc H_{\mf d}}=\max\{\mf d,\mf d_{\mc H}\}$. 
	
	If $\alpha\leq 0$ is such that $\mt_{\mc H}F_0,\mt_{\mc H}F_0^\#\leq\alpha$ for some $F_0\in\mc H\setminus\{0\}$, then 
	\[
		\mc H_{(\alpha)}:= \big\{F\in\mc H:\,\mt_{\mc H} F,\, \mt_{\mc H} F^\# \leq\alpha\big\}\in\Sub^*\mc H
		\,,
	\]
	and we have $\mt_{\mc H}\mc H_{(\alpha)}=\alpha$. 

	Those dB-subspaces which are defined by mean type conditions will in general not exhaust all of $\Sub^*\mc H$. 
	However, sometimes, this also might be the case. 
\end{example}

Trivially, the set $\Sub\mc H$, and hence also each of the sets $\Sub_{\mf d}\mc H$, is partially ordered 
with respect to set-theoretic inclusion. One of the most fundamental and deep results in the theory of de~Branges
spaces is the \emph{Ordering Theorem for subspaces of $\mc H$}, cf.\ \cite[Theorem 35]{debranges:1968} where even a somewhat 
more general version is proved. 

\begin{ntheorem}[De~Branges' Ordering Theorem]{A35}
	Let $\mc H$ be a de~Branges space and let $\mf d:\bb C\to\bb N_0$. Then $\Sub_{\mf d}\mc H$ 
	is totally ordered. 
\end{ntheorem}

The chains $\Sub_{\mf d}\mc H$ have the following continuity property: 
For a dB-subspace $\mc L$ of $\mc H$, set 
\begin{equation}\label{A36}
	\breve{\mc L}:= \bigcap\big\{\mc K\in\Sub_{\mf d_{\mc L}}\mc H:\,\mc K\supsetneq\mc L\big\},\quad
	\text{ if }\mc L\neq\mc H
	\,,
\end{equation}
\[
	\tilde{\mc L}:= \clos_{\mc H}\bigcup\big\{\mc K\in\Sub_{\mf d_{\mc L}}\mc H:
	\,\mc K\subsetneq\mc L\big\},\quad\text{ if }\dim\mc L>1
	\,.
\]
Then 
\[
	\dim\big(\breve{\mc L}/\mc L\big)\leq 1\quad\text{ and }
	\quad\dim\big(\mc L/\tilde{\mc L}\big)\leq 1
	\,.
\]

\begin{example}{A27}
	Let us explicitly mention two examples of de~Branges spaces, which show in some sense extreme behaviour. 
	\begin{enumerate}[$(i)$]
	\item Consider the \emph{Paley--Wiener space} $\PW_a$ where $a>0$. This space is a de~Branges space. 
		It can be obtained as $\mc H(E)$ with $E(z)=e^{-iaz}$. The chain $\Sub^*(\PW_a)$ is equal to 
		\[
			\Sub^*\PW_a=\big\{\PW_b:\,0< b\leq a\big\}
			\,.
		\]
		Apparently, we have $\PW_b=(\PW_a)_{(b-a)}$, and hence in this example all dB-subspaces are obtained by 
		mean type restrictions. 
	\item In the study of the indeterminate Hamburger moment problem de~Branges spaces 
		occur which contain the set of all polynomials $\bb C[z]$ as a dense linear subspace, see e.g.\ 
		\cite{baranov:2006}, \cite{borichev.sodin:1998}, \cite[\S5.9]{dym.mckean:1976}
	
		If $\mc H$ is such that $\mc H=\clos_{\mc H}\bb C[z]$, then the chain $\Sub^*\mc H$ has order type 
		$\bb N$. In fact, 
		\[
			\Sub^*\mc H=\big\{\bb C[z]_n:\,n\in\bb N_0\big\}\cup\{\mc H\}
			\,,
		\]
		where $\bb C[z]_n$ denotes the set of all polynomials whose degree is at most $n$. 
	\end{enumerate}
	Examples of de~Branges spaces $\mc H$ for which the chain $\Sub^*\mc H$ has 
	all different kinds of order types can be constructed using 
	canonical systems of differential equations, see e.g.\ \cite[Theorems 37,38]{debranges:1968}, \cite{gohberg.krein:1970}, or 
	\cite{hassi.desnoo.winkler:2000}. 
\end{example}

With help of the estimates \eqref{A42}, it is easy to see that 
\[
	\mt\frac{K_{\mc H}(w,\cdot)}{\nabla_{\!\mc H}(z)}=0
	\,.
\] 
This implies that for any dB-subspace $\mc L$ of $\mc H$ the supremum in the definition 
of $\mt_{\mc H}\mc L$ is attained (e.g. 
on the reproducing kernel functions $K_{\mc L}(w,\cdot)$). Moreover, 
we obtain $\mt_{\mc H}\mc L=\mt_{\mc H}\nabla_{\!\mc L}$. 

Also the fact whether a given de~Branges space $\mc L$ is contained in $\mc H$ as a dB-subspace can be 
characterized via generating Hermite--Biehler functions. Choose $E,E_1\in\HB$ with $\mc H=\mc H(E)$ and $\mc L=\mc H(E_1)$. 
Then $\mc L\in\Sub^*\mc H$ if and only if there exists a $2\times 2$-matrix function $W(z)=(w_{ij}(z))_{i,j=1,2}$ such that 
the following four conditions hold:
\begin{enumerate}[$(i)$]
\item The entries $w_{ij}$ of $W$ are entire functions, satisify $w_{ij}^\#=w_{ij}$, and $\det W(z)=1$. 
\item The kernel 
	\[
		K_W(w,z):= \frac{W(z)JW(w)^*-J}{z-\qu w}
		\,, \qquad
		J:= 
		\begin{pmatrix}
			0 & -1\\
			1 & 0
		\end{pmatrix}
		\,,
	\]
	is positive semidefinite. 
\item Write $E=A-iB$ and $E_1=A_1-iB_1$ according to \eqref{A16}. Then 
	\[
		(A,B)=(A_1,B_1)W
		\,.
	\]
\item Denote by $\mc K(W)$ the reproducing kernel Hilbert space of $2$-vector functions generated by the kernel 
	$K_W(w,z)$. Then there exists no constant function $\binom uv\in\mc K(W)$ with $uA_1+vB_1\in\mc H(E_1)$. 
\end{enumerate}
Assuming that $\mc L\in\Sub^*\mc H$, the orthogonal complement of $\mc L$ in $\mc H$ can be described via the 
above matrix function $W$. In fact, the map $\binom{f_+}{f_-}\mapsto f_+A_1+f_-B_1$ is an 
isometric isomorphism of $\mc K(W)$ onto $\mc H\ominus\mc L$. 

Let us remark in this context that a function $uA+vB$ can also be written in the form 
$\lambda\cdot(e^{i\psi}E+e^{-i\psi}E^\#)$ and vice versa.
A detailed discussion of the situation when such functions belong
to $\mc H(E)$ and a criterion in terms of the zeros of $E$ can be found e.g.\ in \cite{baranov:2001}. 

Let us discuss in a bit more detail the particular situation that $\mc L\in\Sub^*\mc H$ and 
$\dim(\mc H/\mc L)=1$, cf.\ \cite[Theorem 29, Problem 87]{debranges:1968}. 
In this case $\mc L=\clos_{\mc H}(\dom S_{\mc H})$. Choose $E,E_1\in\HB$ with $\mc H=\mc H(E)$ and 
$\mc L=\mc H(E_1)$. Then the matrix $W$ introduced in the above item is a linear polynomial of the form 
\[
	W(z)=
	\begin{pmatrix}
		1-lz\cos\phi\sin\phi & lz\cos^2\phi\\
		-lz\sin^2\phi & 1+lz\cos\phi\sin\phi
	\end{pmatrix}
	\cdot M
	\,,
\]
where $\phi\in\bb R$, $l>0$, and $M$ is a constant $2\times 2$-matrix with real entries and determinant $1$. 
The space $\mc K(W)$ is spanned by the constant function $\binom{\cos\phi}{\sin\phi}$. We see that 
\[
	\mc H\ominus\mc L=\spn\Big\{(A_1,B_1)\binom{\cos\phi}{\sin\phi}\Big\}=
	\spn\Big\{(A,B)M^{-1}\binom{\cos\phi}{\sin\phi}\Big\}
	\,.
\]

\section{Admissible majorants in de~Branges spaces}

In what follows we will work with functions $\mf m$ which are defined on subsets of the closed upper 
half-plane and take nonnegative real values. To simplify notation we will often drop explicit notation of 
the domain of definition of the function $\mf m$. 

If $\mf m_1,\mf m_2:D\to [0,\infty)$, we write $\mf m_1\lesssim\mf m_2$ if there exists a positive constant 
$C$, such that $\mf m_1(z) \leq C \mf m_2(z)$, $z\in D$. Moreover, $\mf m_1\asymp\mf m_2$ stands for 
"$\mf m_1\lesssim\mf m_2$ and $\mf m_2\lesssim\mf m_1$". Using this notation, we can write 
\[
	R_{\mf m}(\mc H)=\big\{\,F\in\mc H:\,|F(z)|,|F^\#(z)|\lesssim \mf m(z),
	\ z\in D\,\big\}
	\,,
\]
compare with \deref{A1}. 
Our first aim is to show that $\mc R_{\mf m}(\mc H)$ will become a de~Branges subspace of $\mc H$, whenever 
$R_{\mf m}(\mc H)\neq\{0\}$, and $\mf m$ satisfies an obvious regularity condition. 

\begin{theorem}{A4}
	Let $\mc H$ be a de~Branges space, and let $\mf m:D\to[0,\infty)$ be a function with $D\subseteq\bb C^+\cup\bb R$. 
	Then $\mc R_{\mf m}(\mc H)\in\Sub\mc H$ if and only if $\mf m$ satisfies 
	\begin{axioms}{16mm}
	\item[Adm1] $\supp\mf d_{\mf m}\subseteq\bb R$;
	\item[Adm2] $R_{\mf m}(\mc H)$ contains a nonzero element.
	\end{axioms}
	In this case we have 
	\begin{equation}\label{A9}
		\mf d_{\mc R_{\mf m}(\mc H)}=\max\{\mf d_{\mf m},\mf d_{\mc H}\}\quad
		\text{ and } \quad
		\mt_{\mc H}\mc R_{\mf m}(\mc H)\leq\mt_{\mc H} \mf m
		\,.
	\end{equation}
\end{theorem}

\noindent
Necessity of the conditions (Adm1) and (Adm2) is easy to see. 
In the proof of sufficiency we will employ the following elementary lemma. 

\begin{lemma}{A3}
	Let $\mc H$ be a de~Branges space and let $L$ be a nonzero linear subspace of $\mc H$ such that:
	\begin{enumerate}[$(i)$]
	\item if $F\in L$, then also $F^\#\in L$;
	\item if $F\in L$ and $z_0\in\bb C\setminus\bb R$ 
	with $F(z_0)=0$, then also $\frac{F(z)}{z-z_0}\in L$.
	\end{enumerate}
	Then $\clos_{\mc H}L\in\Sub\mc H$. 
\end{lemma}
\begin{proof}
	The mapping $F \mapsto  F^\#$ is continuous on $\mc H$. We have $L^\#\subseteq L$, and hence 
	$(\clos_{\mc H}L)^\#\subseteq\clos_{\mc H}(L^\#)\subseteq\clos_{\mc H}L$. 
	
	Let $F\in\clos_{\mc H}L$, $z_0\in\bb C\setminus\bb R$ with $F(z_0)=0$, be given. We have
	to show that
	\[
		\frac{z-\qu{z_0}}{z-z_0}F(z)\in\clos_{\mc H}L
		\,,
	\]
	or, equivalently, that $\frac{F(z)}{z-z_0}\in\clos_{\mc H}L$.  
	Choose an element $F_0\in L$ with $F_0(z_0)=1$. Such a choice 
	is possible by $(ii)$. The mapping $\rho_{F_0,z_0} : F\mapsto\frac{F(z)-F(z_0)F_0(z)}{z-z_0}$ 
	is continuous. Moreover, by $(ii)$, we have $\rho_{F_0,z_0} L\subseteq L$. Thus
	\[
		\rho_{F_0,z_0}\big(\clos_{\mc H}L\big)\subseteq\clos_{\mc H}(\rho_{F_0,z_0} 
		L)\subseteq \clos_{\mc H}L
		\,.
	\]
	In particular, if $F\in\clos_{\mc H}L$ and $F(z_0)=0$, then $\frac{F(z)}{z-z_0}\in\clos_{\mc H}L$. 
	
	Together, we conclude that $\clos_{\mc H}L\in\Sub\mc H$. 
\end{proof}

\begin{proofof}{\thref{A4}}
	Assume first that $\mc R_{\mf m}(\mc H)\in\Sub\mc H$. Then, clearly, $R_{\mf m}(\mc H)\neq\{0\}$, i.e.\ 
	(Adm2) holds. 
	Let $w\in\bb C$, and choose $F\in R_{\mf m}(\mc H)$ with $\mf d_F(w)=\mf d_{\mc R_{\mf m}(\mc H)}(w)$. 
	By analyticity, we have for some disk $U$ centred at $w$,
	\[
		\inf_{z\in U}\Big|\frac{F(z)}{(z-w)^{\mf d_F(w)}}\Big|>0
		\,.
	\]
	Since $|F(z)|\lesssim\mf m(z)$, $z\in U\cap D$, we obtain that 
	$\mf d_{\mf m}(w)\leq \mf d_F(w)=\mf d_{\mc R_{\mf m}(\mc H)}(w)$.
	It follows that $\mf d_{\mf m}$ takes only finite values and that $\supp\mf d_{\mf m}$ is a discrete subset of $\bb R$. 
	In particular (Adm1) holds. Moreover, we see that 
	\[
		\mf d_{\mc R_{\mf m}(\mc H)} \geq \max\{\mf d_{\mf m},\mf d_{\mc H}\}
		\,.
	\]
	For the converse assume that $\mf m$ satisfies the conditions \textrm{(Adm1)}
	and \textrm{(Adm2)}. We will apply \leref{A3} with
	$L:= R_{\mf m}(\mc H)$. The hypothesis $(i)$ of \leref{A3} is satisfied by the definition of $R_{\mf m}(\mc H)$.
	Let $F\in R_{\mf m}(\mc H)$, $w\in\bb C$, and assume that $\mf d_F(w)>
	\max\{\mf d_{\mf m}(w),\mf d_{\mc H}(w)\}$.
	Then $\frac{F(z)}{z-w}\in\mc H$.
	Let $U$ be a compact neighbourhood of $w$ such that
	\[
		\inf_{\substack{z\in U\cap D\\ |z-w|^ {\mf d_{\mf m}(w)}
		\neq 0}}\frac{\mf m(z)}{|z-w|^{\mf d_{\mf m}(w)}}>0
		\,.
	\]
	For $z\not\in U$, we have $|z-w| \gtrsim 1$, and hence
	$ |z-w|^{-1} |F(z)| \lesssim |F(z)|\lesssim \mf m(z)$, $z\in D\setminus U$.
	The function $ (z-w)^{-\mf d_{\mf m}(w)-1} F(z)$ 
	is analytic, and hence bounded, on $U$. It follows that
	\[
		\Big|\frac{F(z)}{(z-w)^{\mf d_{\mf m}(w)+1}}\Big|
		\lesssim\frac{\mf m(z)}{|z-w|^{\mf d_{\mf m}(w)}},\quad
		z\in
		\begin{cases}
			(D\cap U)\setminus\{w\}, &\ \mf d_{\mf m}(w)>0,\\
			D\cap U, &\ \mf d_{\mf m}(w)=0,
		\end{cases}
	\]
	and hence
	\[
		\Big|\frac{F(z)}{z-w}\Big|\lesssim\mf m(z),\qquad  z\in U\cap D
		\,.
	\]
	The same argument will show that $\big|(\frac{F(z)}{z-w})^\#\big|
	\lesssim\mf m(z)$, $z\in D$, and hence 
	$\frac{F(z)}{z-w}\in R_{\mf m}(\mc H)$.
	Since, by (Adm1), $\mf d_{\mf m}(w)=\mf d_{\mc H}(w)=0$ for
	$w\in\bb C\setminus \bb R$, 
	we conclude that $R_{\mf m}(\mc H)$ satisfies the hypothesis
	$(ii)$ of \leref{A3}. Moreover,
	\[
		\min_{F\in R_{\mf m}(\mc H)\setminus\{0\}} \mf d_F(w)\leq\max\{\mf d_{\mf m}(w),\mf d_{\mc H}(w)\}
		\,.
	\]
	Hence $\mc R_{\mf m}(\mc H)\in\Sub(\mc H)$, and $\mf d_{\mc R_{\mf m}(\mc H)}\leq
	\max\{\mf d_{\mf m},\mf d_{\mc H}\}$.
	We have proved the asserted equivalence and equality of divisors in \eqref{A9}. 
		
	Let $F\in R_{\mf m}(\mc H)$. Then $|F(z)|\lesssim\mf m(z)$, $z\in D$, and hence $\mt_{\mc H} F\leq\mt_{\mc H}\mf m$. 
	This proves the assertion concerning mean types. 
\end{proofof}

\thref{A4} justifies the following definition. 

\begin{definition}{A2}
	Let $\mc H$ be a de~Branges space. A function $\mf m:D\to[0,\infty)$ where $D\subseteq\bb C^+\cup\bb R$, is called an 
	\emph{admissible majorant for $\mc H$} if it satisfies the conditions (Adm1) and (Adm2) of \thref{A4}. 
	
	The set of all admissible majorants is denoted by $\Adm\mc H$. For the set of all those admissible majorants 
	which are defined on a fixed set $D$, we write $\Adm_D\mc H$. 
\end{definition}

\begin{remark}{A5}
	\hspace*{0pt}
	\begin{enumerate}[$(i)$]
	\item We allow majorization on a subset of the closed upper half-plane. Of course, the same definitions could be 
		made for majorants $\mf m$ defined on just any subset of the complex plane. However, due to the symmetry 
		with respect to the real line which is included into the definition of $R_{\mf m}$, this would not be a gain 
		in generality. 
	\item Majorants defined on bounded sets give only trivial results: If $D\subseteq\bb C^+\cup\bb R$ is bounded and 
		$\mf m\in\Adm_D\mc H$, then $R_{\mf m}(\mc H)=\mc R_{\mf m}(\mc H)=
		\mc H_{\mf d_{\mf m}}$.	In view of this fact, we shall once 
                and for all exclude bounded sets $D$ from our considerations. 
	\end{enumerate}
\end{remark}

\begin{remark}{A59}
	As we have already noted after the definition of $\mf d_{\mf m}$, we have $\mf d_{\mf m}(w)=0$ whenever $w\not\in\qu D$. 
	The first formula in \eqref{A9} hence gives 
	\[
		\mf d_{\mc R_{\mf m}(\mc H)}(x)=\mf d_{\mc H}(x),\quad x\in\bb R\setminus\qu D
		\,.
	\]
	It is worth to notice that this statement can also be read in a slightly different way: 
	Assume that $\mc L\in\Sub\mc H$ is represented as $\mc L=\mc R_{\mf m}(\mc H)$ with some $\mf m\in\Adm_D\mc H$. 
	Then $\mf d_{\mc L}|_{\bb R\setminus\qu D}=\mf d_{\mc H}$. 
	
	Assume that $w\in\bb R\setminus\qu D$ and set 
	\[
		\mf d(z)=
		\begin{cases}
			0, &\ z\neq w,\\
			\mf d_{\mc H}(w)+1, &\ z=w.
		\end{cases}
	\]
	Unless $\dim\mc H=1$, the subspace $\mc H_{\mf d}$ will be a dB-subspace of $\mc H$. 
	It follows from the above notice that no subspace $\mc L\in\Sub\mc H$ with $\mc L\subseteq\mc H_{\mf d}$
	can be realized as $\mc R_{\mf m}(\mc H)$ with some $\mf m\in\Adm_D\mc H$. 
\end{remark}

Let us provide some standard examples of admissible majorants. We will mostly work with these majorants. 

\begin{example}{A8}
	An obvious, but surprisingly important, example of admissible majorants is provided by 
	the functions $\nabla_{\!\mc L}|_{\bb C^+\cup\bb R}$, $\mc L\in\Sub\mc H$. 
	Since always $\mc L\subseteq R_{\nabla_{\!\mc L}|_{\bb C^+\cup\bb R}}(\mc H)$, 
	(Adm2) is satisfied. Also, it follows that $\mf d_{\nabla_{\!\mc L}}\leq\mf d_{\mc L}$
	and this yields (Adm1). Thus $\nabla_{\!\mc L}|_{\bb C^+\cup\bb R}\in\Adm\mc H$. 

	The function $\nabla_{\!\mc L}$ is actually for several reasons a distinguished admissible majorant. 
	This will be discussed in more detail later (see, also, the 
	forthcoming paper \cite{baranov.woracek:sprm}). 
\end{example}

\begin{example}{A58}
	Other examples of admissible majorants can be constructed from functions associated to the space $\mc H$. 
	Let $S\in\Assoc\mc H$, and assume that $S$ does not vanish identically and does not satisfy 
	$\mf d_S(z)=\mf d_{\mc H}(z)$, $z\in\bb C$. Moreover, let $D\subseteq\bb C^+\cup\bb R$ be such that 
	$\mf m_S(z)\neq 0$ for all $z\in\qu D\setminus\bb R$. Define 
	\[
		\mf m_S(z):= \frac{\max\{|S(z)|,|S^\#(z)|\}}{|z+i|},\qquad z\in\bb C^+\cup\bb R
		\,,
	\]
	then $\mf m_S|_D\in\Adm\mc H$. In fact, we have 
	\[
		\mf d_{\mf m_S|_D}(w)=
		\begin{cases}
			\mf d_S(w), &\ w\in\bb R\cap\qu D,\\
			0, &\ \text{otherwise},
		\end{cases}
	\]
	and, if $z_0\in\bb C$ is such that $\mf d_S(z_0)>\mf d_{\mc H}(z_0)$, then 
	the function $\frac{S(z)}{z-z_0}$ belongs to $R_{\mf m}(\mc H)$. 

	Provided $D$ containes a part of some ray in $\bb C^+$ with positive logarithmic density, we have 
	\[
		\mt_{\mc H}\mc R_{\mf m_S|_D}(\mc H)= \mt_{\mc H} \mf m_S|_D=\max\{\mt_{\mc H}S,\mt_{\mc H}S^\#\}
		\,.
	\]
\end{example}

\begin{example}{A50}
	Let $\mc H(E)$ be a de~Branges space, and let $\mc L\in\Sub\mc H(E)$. Choose $E_1\in\HB$ with $\mc L=\mc H(E_1)$. 
	Then $E_1(z)\in\Assoc\mc H(E)$ and the conditions required in \exref{A58} are fullfilled for $E_1$. 
	Since each kernel function $K_{\mc L}(w,\cdot)$ of the space $\mc L$ 
        belongs to $R_{\mf m_{E_1}}(\mc H)$, we always have 
        $\mc L\subseteq\mc R_{\mf m_{E_1}|_{\bb C^+}}(\mc H)$. 
	
	Note that the space $R_{\mf m_{E_1}}(\mc H)$ does not depend on the particular choice of $E_1$ with 
	$\mc L=\mc H(E_1)$. In fact, if 
	$E_1$ and $E_2$ both generate the space $\mc L$, then we will have $|E_1(z)|\asymp|E_2(z)|$ 
	throughout $\bb C^+\cup\bb R$, and hence also $\mf m_{E_1}\asymp\mf m_{E_2}$. 
\end{example}

Let us state explicitly how the majorants $\mf m_{E_1}$ and $\nabla_{\!\mc L}$ are related among each other 
and with the space $\mc L$. By \eqref{A42} we have $\mf m_{E_1}\lesssim\nabla_{\!\mc L}$. Moreover, 
\[
	\begin{array}{ccccc}
		&& \mc R_{\mf m_{E_1}}(\mc H) && \\[-3pt]
		& \parbox{2mm}{\hspace*{1mm}\begin{rotate}{40} $\subseteq$ \end{rotate}} && 
			\hspace*{-1mm}\parbox{2mm}{\vspace*{-3mm}\begin{rotate}{-40} $\subseteq$\end{rotate}} &\\[-3pt]
		\mc L &&&& \hspace*{-2mm}\mc R_{\nabla_{\!\mc L}}(\mc H)\\[-3pt]
		& \parbox{2mm}{\hspace*{0mm}\begin{rotate}{-40} $\subseteq$ \end{rotate}} && 
			\hspace*{-1mm}\parbox{2mm}{\vspace*{3mm}\begin{rotate}{40} $\subseteq$\end{rotate}} &\\[-3pt]
		&& R_{\nabla_{\!\mc L}}(\mc H) && \\
	\end{array}
\]

\section{Majorization on rays which accumulate to $\bm i\infty$}

In this section we consider subspaces which are generated by majorization along rays 
being not parallel to the real axis. The following two statements are our main results in this respect. 

\begin{theorem}{A12}
	Consider $D:= i[h,\infty)$ where $h>0$. Let $\mc H$ be a de~Branges space, and let $\mc L\in\Sub^*\mc H$. 
	Then $\mc L=\mc R_{\nabla_{\!\mc L}|_D}(\mc H)$. 
\end{theorem}

\begin{theorem}{A37}
	Consider $D:= e^{i\pi\beta}[h,\infty)$ where $h>0$ and $\beta\in(0,\frac 12)$. 
	Let $\mc H$ be a de~Branges space, assume that each element of $\mc H$ is of zero type with 
	respect to the order $\rho:= (2-2\beta)^{-1}$, and let $\mc L\in\Sub^*\mc H$. 
	Then $\mc L=\mc R_{\nabla_{\!\mc L}|_D}(\mc H)$. 
\end{theorem}

\begin{remark}{A55}
	\hspace*{0pt}
	\begin{enumerate}[$(i)$]
	\item In both theorems we have $\qu D\cap\bb R=\emptyset$. Hence, the requirement $\mc L\in\Sub^*\mc H$ is 
		necessary in order that $\mc L$ can be represented in the form $\mc R_{\mf m}(\mc H)$ with some $\mf m\in\Adm_D\mc H$, 
		cf.\ \reref{A59}. 
	\item With the completely similar proof, the analogue of \thref{A37} for rays $D$ contained in the second quadrant 
		holds true. 	
	\end{enumerate}
\end{remark}

\noindent
In any case $\mc L\subseteq R_{\nabla_{\!\mc L}|_D}(\mc H)$. Hence, in order to establish the asserted equality 
$\mc L=\mc R_{\nabla_{\!\mc L}|_D}(\mc H)$ in either \thref{A12} or \thref{A37}, 
it is sufficient to show that $R_{\nabla_{\!\mc L}|_D}(\mc H)\subseteq\mc L$. 
For the proof of this fact, we will employ the same method as used in the proof of \cite[Theorem 26]{debranges:1968}. 
Let us recall the crucial construction:

Let $F\in\Assoc\mc H$ and $H\in\mc H\ominus\mc L$ be given. 
If $G\in\Assoc\mc L$, we may consider the function 
\begin{equation}\label{A56}
	\Phi_{F,H}(w):= \bigg(\frac{F(z)-\frac{F(w)}{G(w)}G(z)}{z-w},H(z)\bigg)_{\mc H}
	\ .
\end{equation}
In the proof of \cite[Theorem 26]{debranges:1968} it was shown that this function does not depend on the particular 
choice of $G\in\Assoc\mc L$, is entire, and of zero exponential type. 

First we treat a particular situation. 

\begin{lemma}{A52}
	Consider $D:= e^{i\pi\beta}[h,\infty)$ where $h>0$ and $\beta\in(0,\frac 12]$. 
	Let $\mc H$ be a de~Branges space, let $\mc L\in\Sub^*\mc H$, and assume that $\dim\mc H/\mc L=1$. 
	Then $\mc L=\mc R_{\nabla_{\!\mc L}|_D}(\mc H)$. 
\end{lemma}
\begin{proof}
	Let $E,E_1\in\HB$ be such that $\mc H=\mc H(E)$ and $\mc L=\mc H(E_1)$. 
	Assume for definiteness that the choice of $E$ and $E_1$ is made such that $A:= \frac 12(E+E^\#)\in\mc H$ and 
	$A_1:= \frac 12(E_1+E_1^\#)=A$. Then we have $\mc H=\mc L\oplus\spn\{A_1\}$. 

	Since $\mc L\subseteq R_{\nabla_{\!\mc L}|_D}(\mc H)$, the assertion $\mc R_{\nabla_{\!\mc L}|_D}(\mc H)=\mc L$ 
	is equivalent to $A_1\not\in R_{\nabla_{\!\mc L}|_D}(\mc H)$. 
	Assume on the contrary that $|A_1(z)|\lesssim\nabla_{\!\mc L}(z)$, $z\in D$. We have 
	\[
	\begin{aligned}
		\nabla_{\!\mc L}^2(z)
		& = \frac{|E_1(z)|^2-|E_1^\#(z)|^2}{4\pi\Im z}  
		= \frac{|E_1(z)|+|E_1^\#(z)|}{4\pi}\frac{|E_1(z)|-|E_1^\#(z)|}{\Im z}   \\
		& \leq \frac{|E_1(z)|}{2\pi}\frac{|E_1(z)+E_1^\#(z)|}{\Im z}
		=\frac{|E_1(z)|\cdot|A_1(z)|}{\pi\Im z},
		\qquad z\in\bb C^+
		\,.
	\end{aligned}
	\]
	It follows that $|A_1(z)|\lesssim\frac{|E_1(z)|}{\Im z}$, $z\in D$, and hence 
	\[
		1\lesssim\frac 1{\Im z}\Big|\frac{E_1(z)}{A_1(z)}\Big|=
		\frac 1{\Im z}\Big|1-i\frac{B_1(z)}{A_1(z)}\Big|,\qquad 
		z\in D
		\,.
	\]
	Since $A_1\not\in\mc L$, we know from the proof of \cite[Theorem 22]{debranges:1968} that 
	\[
		\lim_{\substack{|z|\to\infty\\ z\in D}}\frac 1{\Im z}\frac{B_1(z)}{A_1(z)}=0
		\,,
	\]
	and have obtained a contradiction. 
\end{proof}

The general case will be reduced to this special case with the help of the next lemma. 
Recall the notation \eqref{A36}:
\[
	\breve{\mc L}:= \bigcap\big\{\mc K\in\Sub_{\mf d_{\mc L}}\mc H:\,\mc K\supsetneq\mc L\big\}
	\,.
\]

\begin{lemma}{A17}
	Let $\mc H$ be a de Branges space, and let $\mc L\in\Sub^*\mc H$, $\mc L\neq\mc H$. 
	Then 
	\[
		\breve{\mc L}=\mc H\cap\Assoc\mc L
		\,.
	\]
\end{lemma}
\begin{proof}
	Set $\mc M:= \mc H\cap\Assoc\mc L$. First we show that $\mc M\in\Sub^*\mc H$. 
	Since both $\mc H$ and $\Assoc\mc L$ are invariant with respect to $F\mapsto F^\#$ and 
	with respect to division by Blaschke factors, also $\mc M$ has this property. The crucial point is to show that $\mc M$ is 
	closed in $\mc H$. To this end, let $(F_n)_{n\in\bb N}$ be a sequence of elements of $\mc M$ which converges 
	to some element $F\in\mc H$ in the norm of $\mc H$. Choose $G_0\in\mc L\setminus\{0\}$ and $w\in\bb C^+$, $G_0(w)\neq 0$. 
	Since the difference quotient operator 
	$\rho_{G_0,w}$ is continuous, we have $\rho_{G_0,w}(F_n)\to\rho_{G_0,w}(F)$ in the norm of $\mc H$. 
	However, $F_n\in\Assoc\mc L$ and $G_0\in\mc L$, and therefore $\rho_{G_0,w}(F_n)\in\mc L$. Since $\mc L$ is 
	closed in $\mc H$, we obtain $\rho_{G_0,w}(F)\in\mc L$. The relation 
	\[
		F(z)=(z-w)\rho_{G_0,w}(F)(z)+\frac{F(w)}{G_0(w)}G_0(z)
	\]
	gives $F\in\Assoc\mc L$. We conclude that $\mc M\in\Sub\mc H$. The fact that $\mc L\subseteq\mc M$ yields 
	in particular that $\mc M\in\Sub^*\mc H$. 
	
	Since we chose $G_0\in\mc L$ and $\mc L\subseteq\mc M$, we have $\dom 
        S_{\mc M}=\rho_{G_0,w}(\mc M)$. However, since 
	$\mc M\subseteq\Assoc\mc L$, $\rho_{G_0,w}$ maps $\mc M$ into $\mc L$
	and it follows that $\dom S_{\mc M}\subseteq\mc L$.
	By \cite[Theorem 29]{debranges:1968}, $\dim(\mc M/\clos \dom S_{\mc M}) \leq 1$ 
	and hence $\dim(\mc M/\mc L)\leq 1$. 
	
	In case $\breve{\mc L}=\mc L$, this implies that $\mc M=\mc L$, and hence also the asserted equality $\breve{\mc L}=\mc M$ 
	holds true. Assume that $\breve{\mc L}\supsetneq\mc L$. Then there exist 
	$E_1,\breve E\in\HB$ and a number $l>0$, such that 
	\[
		\mc L=\mc H(E_1),\ \breve{\mc L}=\mc H(\breve E),\quad 
		(\breve A,\breve B)=(A_1,B_1)\begin{pmatrix} 1 & lz \\ 0 & 1\end{pmatrix}.
	\]
	Moreover, with these choices, we have $\breve{\mc L}=\mc L\oplus\spn\{A_1\}$. Since certainly $A_1\in\Assoc\mc L$, 
	we conclude that $\mc M\supsetneq\mc L$. It follows that also in this case $\mc M=\breve{\mc L}$. 
\end{proof}

\begin{proofof}{\thref{A12}}
	Fix $E,E_1\in\HB$ such that $\mc H=\mc H(E)$ and $\mc L=\mc H(E_1)$. Let $F\in R_{\nabla_{\!\mc L}|_D}(\mc H)$ and 
	$H\in\mc H\ominus\mc L$ be given, and consider the function $\Phi_{F,H}$ defined in \eqref{A56}. 
	
	The basic estimate for our argument is obtained by 
	writing out the inner product in the definition of $\Phi_{F,H}$ as an $L^2$-integral, and applying the 
	Schwarz inequality in $L^2(\bb R)$: For $G\in\Assoc\mc L$ and $w\in\bb C\setminus\bb R$ with $G(w)\neq 0$ we have 
	\begin{equation}\label{A47}
	\begin{gathered}
		|\Phi_{F,H}(w)|=\Big|\int_{\bb R}\frac{F(t)-\frac{F(w)}{G(w)}G(t)}{t-w}\qu{H(t)}\cdot\frac{dt}{|E(t)|^2}\Big|
			\\
		\leq \Big(\int_{\bb R}\Big|\frac{F(t)}{E(t)}\Big|^2\frac{dt}{|t-w|^2}\Big)^{\frac 12}\|H\|_{\mc H}+
		\Big|\frac{F(w)}{G(w)}\Big|\Big(\int_{\bb R}\Big|\frac{G(t)}{E(t)}\Big|^2\frac{dt}{|t-w|^2}\Big)^{\frac 12}
		\|H\|_{\mc H}
		\,.
	\end{gathered}
	\end{equation}
	Note that the integrals on the right side of this inequality converge, since $F\in\mc H$ and 
	$G\in\Assoc\mc L\subseteq\Assoc\mc H$. 

	\hspace*{0pt}\\[-2mm]\textit{Step 1: $\Phi_{F,H}$ vanishes identically.} 
	Let us use the estimate \eqref{A47} with $G:= E_1$ and $w:= iy$, $y\geq h$. From \eqref{A42} and 
	$F\in R_{\nabla_{\!\mc L}|_D}(\mc H)$, we obtain $\lim_{y\to\infty}|E_1^{-1}(iy)F(iy)|=0$. By the Bounded Convergence 
	Theorem both integrals in \eqref{A47} tend to $0$ if $y\to\infty$. In total, 
	\[
		\lim_{y\to\infty}|\Phi_{F,H}(iy)|=0
		\,.
	\]
	Similar reasoning applies with $G:= E_1^\#$ and $w:= -iy$, $y\geq h$, in \eqref{A47}. It follows that also 
	$\lim_{y\to-\infty}|\Phi_{F,H}(iy)|=0$. Since $\Phi_{F,H}$ is of zero exponential type, we may apply the 
	Phragm\'{e}n--Lindel{\"o}f Principle in the left and right half-planes separately, and conclude that the function 
	$\Phi_{F,H}$ vanishes identically. 

	\hspace*{0pt}\\[-2mm]\textit{Step 2: End of proof.} 	
	Since $F\in R_{\nabla_{\!\mc L}|_D}(\mc H)$ and $H\in\mc H\ominus\mc L$ were arbitrary, we conclude that 
	\[
		\frac{F(z)-\frac{F(w)}{G(w)}G(z)}{z-w}\in\mc L,\qquad F\in R_{\nabla_{\!\mc L}|_D}(\mc H),\ G\in\Assoc\mc L
		\,.
	\]
	This just says that $R_{\nabla_{\!\mc L}|_D}(\mc H)\subseteq\Assoc\mc L$, and \leref{A17} gives 
	$R_{\nabla_{\!\mc L}|_D}(\mc H)\subseteq\breve{\mc L}$. 
	
	If $\mc L=\breve{\mc L}$, we are already done. Otherwise, applying \leref{A52} with the spaces $\breve{\mc L}$ and 
	$\mc L\in\Sub^*\breve{\mc L}$, and using the already established fact $R_{\nabla_{\!\mc L}|_D}(\mc H)\subseteq\breve{\mc L}$, 
	gives 
	\[
		R_{\nabla_{\!\mc L}|_D}(\mc H)\subseteq R_{\nabla_{\!\mc L}|_D}(\breve{\mc L})=\mc L
		\,.
	\]
\end{proofof}

\begin{proofof}{\thref{A37}}
	Again fix $E,E_1\in\HB$ such that $\mc H=\mc H(E)$ and $\mc L=\mc H(E_1)$. Let $F\in R_{\nabla_{\!\mc L}|_D}(\mc H)$ and 
	$H\in\mc H\ominus\mc L$ be given, and consider the function $\Phi_{F,H}$. 

	\hspace*{0pt}\\[-2mm]\textit{Step 1: $\Phi_{F,H}$ is of order $\rho$ and zero type.} The function $E_1$ belongs to $\HB$ 
	and is of order $\rho=(2-2\beta)^{-1}<1$. Hence $|E_1(iy)|$ is a nondecreasing function of $y>0$. Let $\epsilon>0$ be given, 
	then there exists $C>0$ with 
	\[
		|F(z)|\leq Ce^{\epsilon|z|^\rho},\quad z\in\bb C
		\,.
	\]
	Using \eqref{A47} with $G:= E_1$ and $w:= iy$, $y\geq 1$, we obtain 
	\[
		|\Phi_{F,H}(iy)|\leq\Big(\int_{\bb R}\Big|\frac{F(t)}{E(t)}\Big|^2\frac{dt}{t^2+1}\Big)^{1/2}\cdot\|H\|_{\mc H}
	\]
	\[
		+\frac{Ce^{\epsilon y^\rho}}{|E_1(i)|}\cdot\Big(\int_{\bb R}\Big|\frac{E_1(t)}{E(t)}\Big|^2\frac{dt}{t^2+1}\Big)^{1/2}
		\cdot\|H\|_{\mc H}=O(e^{\epsilon y^\rho}),\quad y\geq 1
		\,.
	\]
	Using $G:= E_1^\#$ instead of $E_1$, gives the analogous estimate $|\Phi_{F,H}(iy)|=O(e^{\epsilon |y|^\rho})$ for $y<-1$. 
	In total, we have $|\Phi_{F,H}(iy)|=O(e^{\epsilon |y|^\rho})$, $y\in\bb R$. Applying the 
	Phragm\'{e}n--Lindel{\"o}f Principle to 
	the left and right half-planes separately, yields that $\Phi_{F,H}$ is of order $\rho$. Since $\epsilon>0$ was arbitrary, $\Phi_{F,H}$ 
	is of zero type with respect to this order. 

	\hspace*{0pt}\\[-2mm]\textit{Step 2: $\Phi_{F,H}$ vanishes identically.} Let $C>0$ be such that 
	$|F(z)|,|F^\#(z)|\leq C\mf m_{\nabla_{\!\mc L}|_D}(z)$, 
	$z\in D$. We obtain from \eqref{A42} and the estimate \eqref{A47} used with $G:= E_1$ and $w\in D$, that 
	\begin{multline*}
		|\Phi_{F,H}(w)|\leq \Big(\int_{\bb R}\Big|\frac{F(t)}{E(t)}\Big|^2\frac{dt}{|t-w|^2}\Big)^{1/2}\cdot\|H\|_{\mc H}+
			\\
		+\frac C{2\sqrt{\pi\Im w}}\Big(\int_{\bb R}\Big|\frac{E_1(t)}{E(t)}\Big|^2\frac{dt}{|t-w|^2}\Big)^{1/2}
		\cdot\|H\|_{\mc H},\quad w\in D
		\,.
	\end{multline*}
	Since $\int_{\bb R}|E^{-1}F(t)|^2(1+t^2)^{-1}\,dt,\int_{\bb R}|E^{-1}E_1(t)|^2(1+t^2)^{-1}\,dt<\infty$, and 
	$\beta\in(0,\frac 12)$, both integrals tend to zero if $|w|\to\infty$ within $D$. We obtain that 
	\[
		\lim_{\substack{|w|\to\infty\\ w\in D}}|\Phi_{F,H}(w)|=0
		\,.
	\]
	The similar argument, applying \eqref{A47} with $G:= E_1^\#$ and $\qu w\in D$, will give 
	$\lim_{\substack{|w|\to\infty\\ w\in D}}|\Phi_{F,H}(w)|=0$. 

	Consider the region $\mc G_+$ which is bounded by the ray $D$, its conjugate ray, and the line segment 
	connecting these rays. Moreover, let $\mc G_-:= \bb C\setminus\qu{\mc G_+}$. 
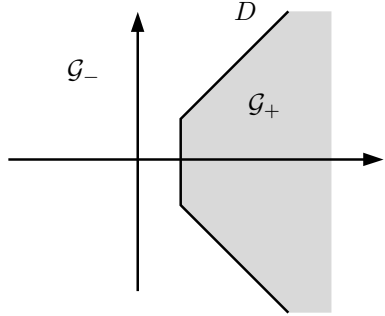
\begin{center}
\begin{picture}(0,0)%
\includegraphics{figure.pstex}%
\end{picture}%
\setlength{\unitlength}{4144sp}%
\begingroup\makeatletter\ifx\SetFigFont\undefined%
\gdef\SetFigFont#1#2#3#4#5{%
  \reset@font\fontsize{#1}{#2pt}%
  \fontfamily{#3}\fontseries{#4}\fontshape{#5}%
  \selectfont}%
\fi\endgroup%
\begin{picture}(2291,1848)(6838,-2305)
\put(8278,-530){\makebox(0,0)[b]{\smash{{\SetFigFont{10}{6.0}{\rmdefault}{\mddefault}{\updefault}{$D$}%
}}}}
\put(7312,-878){\makebox(0,0)[b]{\smash{{\SetFigFont{10}{6.0}{\rmdefault}{\mddefault}{\updefault}{$\mathcal G_-$}%
}}}}
\put(8394,-1084){\makebox(0,0)[b]{\smash{{\SetFigFont{10}{6.0}{\rmdefault}{\mddefault}{\updefault}{$\mathcal G_+$}%
}}}}
\end{picture}%
\end{center}
	The opening of $\mc G_+$ is $2\pi\beta<\pi<\rho^{-1}\pi$. The opening of $\mc G_-$ is $\pi(2-2\beta)=\rho^{-1}\pi$. 
	Since $\Phi_{F,H}$ is of order $\rho$ zero type, we may apply the Phragm\'{e}n--Lindel{\"o}f Principle to the regions 
	$\mc G_+$ and $\mc G_-$ separately, and conclude that $\Phi_{F,H}$ vanishes identically. 

	\hspace*{0pt}\\[-2mm]\textit{Step 3: End of proof.} We repeat word by word the same reasoning as in Step 2 of the 
	proof of \thref{A12}, and obtain the desired assertion. 
\end{proofof}

\begin{remark}{A39}
	Consider $D:= e^{i\pi\beta}[h,\infty)$ where $h>0$ and $\beta\in(0,\frac 12)$. 
	We do not know at present, and find this an intriguing problem, whether \thref{A37} 
	remains valid without any assumptions on the growth of elements of $\mc H$. 
	It is clear where the argument in the above proof breaks: 
	If we merely know that $\Phi_{F,H}$ is of zero exponential type, its smallness on the boundary of $\mc G_+$ 
	does not imply that $\Phi_{F,H}\equiv 0$. 
	
	On the other hand, we were not able to construct a counterexample. 
	One reason for this will be explained later, cf.\ \reref{A53}. 
\end{remark}

Having in mind \thref{A12} and the formula \eqref{A9} for zero divisors, 
it does not anymore come as a surprise that all dB-subspaces of a given de~Branges space can be realized by 
majorization on $D:= \bb R\cup i[0,\infty)$. 

\begin{corollary}{A13}
	Consider $D:= \bb R\cup i[0,\infty)$. Let $\mc H$ be a de~Branges space, and let $\mc L\in\Sub\mc H$. 
	Then $\mc L=\mc R_{\nabla_{\!\mc L}|_D}(\mc H)$. 
\end{corollary}
\begin{proof}
	We have $\mf d_{\nabla_{\!\mc L}}=\mf d_{\mc L}$, and hence 
	$\mc R_{\nabla_{\!\mc L}|_D}(\mc H)\subseteq\mc H_{\mf d_{\mc L}}$. Since $\mc L\in\Sub^*\mc H_{\mf d_{\mc L}}$, 
	we may apply \thref{A12}, and obtain 
	\[
		\mc L\subseteq\mc R_{\nabla_{\!\mc L}|_D}(\mc H)\subseteq\mc R_{\nabla_{\!\mc L}|_D}(\mc H_{\mf d_{\mc L}})
		\subseteq\mc R_{\nabla_{\!\mc L}|_{i[1,\infty)}}(\mc H_{\mf d_{\mc L}})
		\subseteq\mc L
		\,.\vspace*{-8mm}
	\]
\end{proof}

Another statement which fits the present context can be proved by a more elementary argument. 
Also this result is not much of a surprise, when 
thinking of \cite[Theorem 3.4]{baranov.woracek:dbmaj} 
(see \thref{A10} below) and the estimate \eqref{A9} for mean type. 

\begin{proposition}{A54}
	Consider $D:= \bb R\cup e^{i\pi\beta}[0,\infty)$ where $\beta\in(0,\frac 12]$. Let $\mc H$ be a de~Branges space, 
	and let $\mc L=\mc H(E_1)\in\Sub\mc H$. Then $\mc L=\mc R_{\mf m_{E_1}|_D}(\mc H)$. 
\end{proposition}
\begin{proof}
	In any case, we have $\mc L\subseteq\mc R_{\mf m_{E_1}|_D}(\mc H)$, cf.\ \exref{A50}. Hence, in order to 
	establish the asserted equality, it suffices to show that $R_{\mf m_{E_1}|_D}(\mc H)\subseteq\mc L$. 
	
	Let $F\in R_{\mf m_{E_1}|_D}(\mc H)$ be given, so that $|E_1^{-1}(z)F(z)|\lesssim|z+i|^{-1}$, 
	$z\in\bb R\cup e^{i\pi\beta}[0,\infty)$. Since $F\in\mc H$, the quotient $E_1^{-1}F$ is of bounded type 
	in $\bb C^+$. Since $F$ is majorized by $\mf m_{E_1}$ 
	on the ray $e^{i\pi\beta}[0,\infty)$, 
	we have $\mt E_1^{-1}F\leq 0$. The same arguments apply 
	to $F^\#$. Finally, since $F$ is majorized along the real axis, we have 
	$E_1^{-1}F\in L^2(\bb R)$. It follows that $F\in\mc H(E_1)=\mc L$. 
\end{proof}

\begin{remark}{A32}
	We would like to point out that, although seemingly very similar, \prref{A54} differs in some essential points 
	from the previous results \thref{A12}, \thref{A37} and \coref{A13}. The obvious differences are of course 
	that on the one hand $\mf m_{E_1}|_D\lesssim\nabla_{\!\mc L}|_D$, but on the other hand also in case $\beta\in(0,\frac 12)$ 
	there are no growth assumptions on $\mc H$. 
	
	The following two notices are not so obvious. First, in \prref{A54} 
	we obtain only $\mc L=\clos_{\mc H}R_{\mf m_{E_1}|_D}(\mc H)$, and taking the closure is in general necessary. 
	In the previous statements, we had $\mc L=\mc R_{\nabla_{\!\mc L}|_D}(\mc H)$ and hence actually 
	$\mc L=R_{\nabla_{\!\mc L}|_D}(\mc H)$, cf.\ the chain of inclusions in \exref{A50}. Secondly, the argument used to 
	prove \prref{A54} relies mainly on majorization along $\bb R$; majorization along the ray is only used to 
	control mean type. Contrasting this, the argument used to deduce \coref{A13} from \thref{A12} relies mainly 
	on majorization on the ray; majorization along $\bb R$ is only used to control $\mf d_{\mc L}$. 
\end{remark}

\section{Majorization on sets close to $\bb R$}

In this section, we focus on majorization on sets $D$ which are close to the real axis.
If majorization is permitted only on $\bb R$ itself, we already know precisely which subspaces can be represented. 
Recall: 

\begin{theorem}[\parbox{31mm}{\cite[Theorem 3.4]{baranov.woracek:dbmaj}}]{A10}
	Consider $D:= \bb R$. Let $\mc H$ be a de~Branges space, and let $\mc L=\mc H(E_1)\in\Sub\mc H$. 
	If $\mt_{\mc H}\mc L=0$, then $\mc L=\mc R_{\mf m_{E_1}|_D}(\mc H)$. Conversely, if $\mc L=\mc R_{\mf m}(\mc H)$ 
	with some $\mf m\in\Adm_{\bb R}\mc H$, then $\mt_{\mc H}\mc L=0$. 
\end{theorem}

In view of this result, it is not surprising that we cannot capture mean type restrictions 
if $D$ is too close to the real line. 
The following statement makes this quantitatively more precise. 

\begin{theorem}{A11}
	Let $\mc H$ be a de~Branges space, and let $\psi:\bb R\to\bb R$ be a positive and even function which is increasing 
	on $[0,\infty)$ and satisfies 
	\[
		\int_0^\infty\frac{\psi(t)}{t^2+1}\,dt<\infty
		\,.
	\]
	If $D\subseteq\{z\in\bb C^+\cup\bb R:\,\Im z\leq \psi(\Re z)\}$ and $\mf m\in\Adm_D\mc H$, then 
	$\mt_{\mc H}\mc R_{\mf m}(\mc H)=0$. 
\end{theorem}
\begin{proof}
	Let $E$ be such that $\mc H=\mc H(E)$, pick $F\in R_{\mf m}(\mc H)$, $\|F\|_{\mc H}=1$, and set $a:= \mt_{\mc H}F$. 
	If $a=0$, we are already done, hence assume that $a<0$. Let $\varepsilon\in(0,-\frac a2)$, and let 
	$\tilde\psi:\bb R\to\bb R$ be a positive and even function which is increasing on $[0,\infty)$, satisfies 
	$\psi(x)=o(\tilde \psi(x))$, $x\to+\infty$, and is such that still 
	\[
		\int_0^\infty \frac{\tilde\psi(t)}{t^2+1}\,dt < \infty
		\,.
	\]
	It is a well-known consequence of the Beurling--Malliavin Theorem, that there exists 
	a nonzero function $f\in\PW_\varepsilon$ with 
	\[
		|f(x)|\leq\exp(-\tilde\psi(x)),\qquad x\in \bb R
		\,,
	\]
	see e.g.\ \cite[p.276]{havin.joericke:1994} or \cite[p.159]{koosis:1992}. By the Phragm\'{e}n--Lindel{\"o}f Principle the 
	functions $e^{i\varepsilon z}f$ and $e^{i\varepsilon z}f^\#$ are bounded by $1$ throughout $\bb C^+$. 
	
	Since $\tilde\psi$ is even and increasing on $[0,\infty)$, we can estimate the Poisson integral 
	for $\log|f|$ to obtain 
	\[
	     |f(z)|\leq\exp\big(\varepsilon |\Im z|-C\tilde\psi(|z|)\big), \qquad z \in \bb C
	     \,,
	\]
	where the constant $C>0$ does not depend $z$. 
	
	Consider the function 
	\[
		G(z):= F(z)f(z)e^{i(a+\varepsilon)z}
		\,.
	\]
	Then $|G(x)|\le |F(x)|$, $x\in \bb R$, and 
	\begin{align}
		&\mt\frac{G}{E}=\mt\frac{F}{E}+\mt f-(a+\varepsilon)=\mt f-\varepsilon\leq 0,
			\label{A34} \\ \nonumber
		&\mt\frac{G^\#}{E}=\mt\frac{G^\#}{E}+\mt f^\#+(a+\varepsilon)\leq a+2\varepsilon <0
			\,.
	\end{align}
	We conclude that $G\in \mc H$. Next we show that $G\in R_{\mf m}(\mc H)$. Indeed, if $z\in D$ or $\qu z\in D$, $z=x+iy$, 
	then $|y|\leq\psi(x)$. Since $\psi(x)=o(\tilde\psi(x))$, $x\to+\infty$, we may estimate 
	\begin{multline*}
		|G(z)|\leq |F(z)|e^{(-a+2\varepsilon) |y|-C\tilde \psi(|z|)}\leq
			\mf m(z)e^{(-a+2\varepsilon)\psi(x)-C\tilde \psi(x)}\lesssim \mf m(z)\,,\\
		z\in D\text{ or }\qu z\in D\,.
	\end{multline*}
	Finally, by the Phragm\'{e}n--Lindel{\"o}f Principle, 
	$\mt f\geq-\varepsilon$. Using \eqref{A34}, it follows that 
	$\mt E^{-1}G\geq-2\varepsilon$ and so 
	$\mt_{\mc H}\mc R_{\mf m}(\mc H)\geq-2\varepsilon$. 
	Since $\varepsilon\in(0,-\frac a2)$ was arbitrary, 
	we conclude that $\mt_{\mc H}\mc R_{\mf m}(\mc H)=0$. 
\end{proof}

\thref{A11} can, of course, 
be viewed as a necessary condition for representability of a dB-subspace $\mc L$ as $\mc R_{\mf m}(\mc H)$, 
where $\mf m$ is defined on some set $D$ close to $\bb R$. 
We turn to the discussion of sufficient conditions for representability. To start with, let us discuss representability 
with the standard majorant $\nabla_{\!\mc L}$ restricted to $\bb R$. 

\begin{theorem}{A15}
	Consider $D:= \bb R$. Let $\mc H$ be a de~Branges space, and let $\mc L=\mc H(E_1)\in\Sub\mc H$. 
	If $\mt_{\mc H}\mc L=0$ and $\sup_{x\in\bb R}\varphi_{E_1}'(x)<\infty$, then 
	\[
		\mc L\subseteq\mc R_{\nabla_{\!\mc L}|_D}(\mc H)\subseteq\breve{\mc L}
		\,. 
	\]
	Thereby $\mc R_{\nabla_{\!\mc L}|_D}(\mc H)\neq\mc L$, if and only if $\breve{\mc L}\supsetneq\mc L$ and 
	for some $\varphi_0\in \bb R$ we have 
	\[
		\big|\cos(\varphi_{E_1}(x)-\varphi_0)\big|
                \lesssim \big(\varphi_{E_1}'(x)\big)^{1/2},\qquad x\in\bb R
		\,.
	\]
\end{theorem}
\begin{proof}
	Let $F\in R_{\nabla_{\!\mc L}|_{\bb R}}(\mc H)$ be given. Since $\mt_{\mc H}\mc L=0$, 
	we have $\mt E_1^{-1}F=\mt E^{-1}F\leq 0$. 
	Thus $E_1^{-1}F\in\mc N_+$. Due to our present assumption on $\varphi_{E_1}'$, we have 
	\[
		|F(x)|\lesssim\nabla_{\!\mc L}(x)=\pi^{-1/2} |E_1(x)| (\varphi_{E_1}'(x))^{1/2} \lesssim |E_1(x)|,
		\qquad x\in\bb R
		\,.
	\]
	The Smirnov Maximum Principle implies that $|F(z)|\lesssim|E_1(z)|$ throughout the half-plane $\bb C^+$. 
	It follows that $F\in\Assoc\mc L$. Using \leref{A17}, we obtain $\mc R_{\nabla_{\!\mc L}|_{\bb R}}(\mc H)\subseteq\breve{\mc L}$. 
	
	By what we just proved, in case $\breve{\mc L}=\mc L$, certainly also $\mc R_{\nabla_{\!\mc L}|_{\bb R}}(\mc H)=\mc L$. 	
	Hence, in order to characterize the situation that $\mc R_{\nabla_{\!\mc L}|_{\bb R}}(\mc H)\neq\mc L$, 
	we may assume that $\breve{\mc L}\neq\mc L$. 
	
	Let $\varphi_0\in\bb R$ be such that $\breve{\mc L}=\mc L\oplus\spn\{S_1\}$, with 
	$S_1:= e^{i\varphi_0}E_1+e^{-i\varphi_0}E_1^\#$. We have 
	\[
		S_1(x)=2|E_1(x)|\cos\big(\varphi_{E_1}(x)-\varphi_0\big)
		\,,
	\]
	and hence $S_1\in R_{\nabla_{\!\mc L}|_{\bb R}}(\mc H)$ if and only if 
	$|\cos(\varphi_{E_1}(x)-\varphi_0)| 
        \lesssim \big( \varphi_{E_1}'(x)\big)^{1/2}$, $x\in\bb R$. Since 
	$\mc R_{\nabla_{\!\mc L}|_{\bb R}}(\mc H)\neq\mc L$ is equivalent to $S_1\in R_{\nabla_{\!\mc L}|_{\bb R}}(\mc H)$, 
	the assertion follows. 
\end{proof}

\begin{corollary}{A19}
	Consider $D:= \bb R$. Let $\mc H=\mc H(E)$ be a de~Branges space, 
	assume that $\sup_{x\in\bb R}\varphi_E'(x)<\infty$ and $\breve{\mc K}=\mc K$, $\mc K\in\Sub^*\mc H$, 
	and let $\mc L\in\Sub^*\mc H$. 
	If $\mt_{\mc H}\mc L=0$, then $\mc L=\mc R_{\nabla_{\!\mc L}|_{\bb R}}(\mc H)$. 
\end{corollary}
\begin{proof}
	By \cite[Problem 154]{debranges:1968}, the function $\varphi_{\mc L}'$, 
	$\mc L\in\Sub^*\mc H$, depends monotonically on $\mc L$. By this we mean that 
	\[
		\varphi_{\mc L}'(x)\leq\varphi_{\mc K}'(x),\quad x\in\bb R,\quad
		\text{ whenever }\quad \mc L,\, \mc K\in\Sub^*\mc H, \ 
		\mc L\subseteq\mc K
		\,.
	\]
	The present assertion is now an immediate consequence of \thref{A15}. 
\end{proof}

Already the following simple remark shows that some additional 
conditions are really needed in order to obtain 
$\mc L=\mc R_{\nabla_{\!\mc L}|_{\bb R}}(\mc H)$.

\begin{remark}{A51}
	Let $\mc H$ be a de~Branges space, let $\mc L\in\Sub\mc H$, and write $\mc L=\mc H(E_1)$. If 
	$\inf_{x\in\bb R}\varphi_{E_1}'(x)>0$, then $\breve{\mc L}\subseteq\mc R_{\nabla_{\!\mc L}|_{\bb R}}(\mc H)$. 
	This follows since the condition $\inf_{x\in\bb R}\varphi_{E_1}'(x)>0$ 
        certainly implies that every linear combination $\lambda E_1+\mu E_1^\#$ 
        is majorized by $\nabla_{\!\mc L}$ on the real axis. 
\end{remark}

However, the situation can be really bad, if $\varphi'_{E_1}$ grows fast.

\begin{example}{A45}
	We are going to construct spaces $\mc H$ and $\mc L\in\Sub^*\mc H$, such that $\dim\mc H/\mc L=\infty$ and 
	$R_{\nabla_{\!\mc L}|_{\bb R}}(\mc H)=\mc H$. 
	
	As we see from the proof of \thref{A15}, 
	it will be a good start to construct the space $\mc L=\mc H(E_1)$ such that $\varphi_{E_1}'(x)$ grows very fast. 
	Let $z_n:= (\sign n)\log |n|+i|n|^{-1}\log^{-2}|n|$, $n\in\bb Z$, 
	$|n|\ge 2$. Since $\sum_{n=1}^\infty|\Im\frac 1{z_n}|<\infty$, 
	there exists a function $E_1\in\HB$ with $E_1(-z)=E_1^\#(z)$ and $\mt E_1^{-1}E_1^\#=0$, 
	which has the points $z_n$ as simple zeros and does not vanish at any other point. For this function, we have 
	\[
		\varphi_{E_1}'(x)=\sum_{\substack{n\in\bb Z\\ |n|\ge 2}} \frac 1{n\log^2 n \Big[(x-\log n)^2+n^{-2}\log^{-4}n\Big]}
		\,.
	\]
	Let $x\geq\ln 2$ be given, and choose $k\in\bb N$, $k\ge 2$, such that $x\in[\log k,\log(k+1))$. 
	Then we have 
	\[
	\begin{aligned}
		k\log^2k\bigg[(x-\log k)^2+\frac 1{k^2\log^4k}\bigg] 
		& \leq 
		k\log^2k\bigg[\big(\underbrace{\log(k+1)-\log k}_{=\log(1+\frac 1k)
		\leq\frac 1k}\big)^2+\frac 1{k^2\log^4k}\bigg] \\
		& \leq\frac{\log^2k}k\bigg[1+\frac 1{\log^4k}\bigg]\leq\frac{\log^2k}k\bigg[1+\frac 1{\log^42}\bigg]
		\,.
	\end{aligned}
	\]
	It follows that 
	\begin{equation}\label{A46}
		\varphi_{E_1}'(x)\geq \big[1+\frac 1{\log^42}\big]^{-1}\frac k{\log^2k}\geq 
		\frac{e^x-1}{x^2}\big[1+\frac 1{\log^42}\big]^{-1},\qquad x\geq\ln 2
		\,.
	\end{equation}
	Next, choose an entire matrix function $W(z)=(w_{ij}(z))_{i,j=1,2}$, $W\neq I$, of zero exponential type, 
	with $w_{ij}^\#=w_{ij}$, $W(0)=I$, $\det W(z)=1$, such that the kernel $K_W(w,z)$ is positive semidefinite and the reproducing 
	kernel space $\mc K(W)$ does not contain a constant vector function. Examples of such matrix functions can be obtained easily 
	using the theory of canonical systems. Define a function $E=A-iB$ by 
	\[
		\big(A(z),B(z)\big):= \big(A_1(z),B_1(z)\big)W(z)
		\,.
	\]
	Then $E\in\HB$ and $\mc H(E_1)\in\Sub^*\mc H(E)$. Moreover, the space $\mc K(W)$ is isomorphic to the orthogonal complement 
	$\mc H(E)\ominus\mc H(E_1)$ via the map 
	\[
		\binom{f_+}{f_-}\mapsto f_+A_1+f_-B_1
		\,.
	\]
	In particular, $\dim(\mc H(E)\ominus\mc H(E_1))=\dim\mc K(W)=\infty$. Note here that $\mc K(W)$ is certainly infinite dimensional, since 
	it does not contain any constant. Finally, note that all elements $\binom{f_+}{f_-}$ of $\mc K(W)$ are of zero exponential type. 
	In view of \eqref{A46} and the symmetry of $E_1$, this implies that 
	\[
		|f_+(x)A_1(x)+f_-(x)B_1(x)|\lesssim 
		|E_1(x)| (\varphi_{E_1}'(x))^{1/2}=\pi^{1/2} 
		\nabla_{\mc H(E_1)}(x),
		\quad x\in \bb R
		\,,
	\]
	i.e.\ $f_+A_1+f_-B_1\in R_{\nabla_{\!\mc H(E_1)}|_{\bb R}}(\mc H(E))$. We conclude that 
	$R_{\nabla_{\!\mc H(E_1)}|_{\bb R}}(\mc H(E))=\mc H(E)$. 
\end{example}

Our next aim is to obtain some information about majorization on lines 
$D:= \bb R+ih$, $h>0$, parallel to the real axis. 
In order that a subspace $\mc L\in\Sub\mc H$ can be represented by majorization on $D$, it is 
not only necessary to have $\mt_{\mc H}\mc L=0$, but also that $\mc L\in\Sub^*\mc H$. 

\begin{theorem}{A18}
	Consider $D:= \bb R+ih$ where $h>0$. Let $\mc H$ be a de~Branges space, and let $\mc L=\mc H(E_1)\in\Sub^*\mc H$. 
	If $\mt_{\mc H}\mc L=0$, then $\mc L=\mc R_{\mf m_{E_1}|_D}(\mc H)$ and 
	$\mc L\subseteq\mc R_{\nabla_{\!\mc L}|_D}(\mc H)\subseteq\breve{\mc L}$. 
\end{theorem}
\begin{proof}
	It is easy to see that, for the proof of the present assertion, we may assume without loss of generality that 
	$\mf d_{\mc H}=0$. 

\hspace*{0pt}\\[-2mm]\textit{Step 1:} Write $\mc H=\mc H(E)$. We prove the following statement: \textit{If 
	$S\in\Assoc\mc H$, and $|E_1^{-1}S|$ and $|E_1^{-1}S^\#|$ are bounded on the line $D=\bb R+ih$, then $S\in\Assoc\mc L$.}

	The function $E_1^{-1}S$ is analytic on a domain containing the closed 
	half-plane $\{z\in\bb C:\,\Im z\geq h\}$ and is of bounded type in the open half-plane 
	$\bb H:= \{z\in\bb C:\,\Im z>h\}$. Also, we have 
	\[
		\mt\frac S{E_1}=\mt\frac SE+\underbrace{\mt\frac E{E_1}}_{=0}\leq 0
		\,.
	\]
	Thus it belongs to the Smirnov class $\mc N_+(\bb H)$ in the half-plane $\bb H$. The Smirnov Maximum Principle hence applies, 
	and we obtain that $E_1^{-1}S$ is bounded throughout $\bb H$. The same argument applies with $S^\#$ in place of $S$. 
	Applying \cite[Theorem 26]{debranges:1968} with the measure $d\mu(t):= |E(t)|^{-2}\,dt$ 
	gives $S\in\Assoc\mc L$. 

\hspace*{0pt}\\[-2mm]\textit{Step 2:} Let $F\in R_{\nabla_{\!\mc L}|_D}(\mc H)$. We have 
	\[
		\nabla_{\!\mc L}(z)=\Big(\frac{|E_1(z)|^2-|E_1^\#(z)|^2}{4\pi h}\Big)^{1/2}\lesssim|E_1(z)|, \qquad z\in D
		\,,
	\]
	and hence $|E_1^{-1}F|$ and $|E_1^{-1}F^\#|$ are bounded on $D$. By Step 1, it follows that $F\in\Assoc\mc L$. 
	\leref{A17} implies that $F\in\breve{\mc L}$, and we conclude that $\mc R_{\nabla_{\!\mc L}|_D}(\mc H)\subseteq\breve{\mc L}$. 

\hspace*{0pt}\\[-2mm]\textit{Step 3:} Let $F\in R_{\mf m_{E_1}|_D}(\mc H)$. The function $S(z):= zF(z)$ is associated to $\mc H$, 
	and $E_1^{-1}S$ as well as $E_1^{-1}S$ are bounded on $D$. By Step 1, $S\in\Assoc\mc L$, 
	and thus $F\in\mc L$. We see that $\mc R_{\mf m_{E_1}|_D}(\mc H)\subseteq\mc L$. The reverse inclusion holds in any case, 
	cf.\ \exref{A50}. 
\end{proof}

It is easy to give an example of de~Branges spaces $\mc H$ and $\mc L\in\Sub^*\mc H$, $\mt_{\mc H}\mc L=0$, such that 
$\mc R_{\nabla_{\!\mc L}|_{\bb R+ih}}(\mc H)\neq\mc L$. 

\begin{example}{A20}
	Let $D:= \bb R+ih$ where $h>0$. Consider the space $\mc H:= \mc H(E)$ generated by the function 
	\[
		E(z):= \cos z-i(z\cos z+\sin z)
		\,.
	\]
	The choice of $E$ is made such that 
	\[
		(A(z),B(z))=(\cos z,\sin z)
		\begin{pmatrix}
			1 & z\\
			0 & 1
		\end{pmatrix}
	\]
	Thus $\mc H(E)$ contains $\mc L:= \PW_1$ as a dB-subspace with codimension $1$, 
	and $\mc H=\mc L\oplus\spn\{\cos z\}$. Since for $z=x+ih\in D$,
	$\nabla_{\PW_1}(z)=\big(\frac{{\rm sh}\, 2h}{2\pi h}\big)^{1/2}$,
	the function $\cos z$ belongs to $R_{\nabla_{\!\mc L}|_D}(\mc H)$. 
	Thus $\mc R_{\nabla_{\!\mc L}|_D}(\mc H)=\mc H$. 
\end{example}

Finally, we turn to majorization on rays parallel to the real axis. 

\begin{theorem}{A48}
	Consider $D:= iy_0+[h,\infty)$ where $h\in\bb R$ and $y_0\geq 0$. 
	Let $\mc H$ be a de Branges space, assume that each element of 
	$\mc H$ is of zero type with respect to the 
	order $\rho:= \frac 12$, and let $\mc L=\mc H(E_1)\in\Sub^*\mc H$. Then 
	$\mc R_{\mf m_{E_1}|_D}(\mc H) = \mc L$.
\end{theorem}
\begin{proof}

	\hspace*{0pt}\\[1mm]\textit{Step 1: The case $y_0>0$.} We proceed similar as in the proof of \thref{A37}. 
	Fix $E,E_1\in\HB$ such that $\mc H=\mc H(E)$ and $\mc L=\mc H(E_1)$. Let $F\in R_{\mf m_{E_1}|_D}(\mc H)$ and 
	$H\in\mc H\ominus\mc L$ be given, and consider the function $\Phi_{F,H}$
	defined as in \eqref{A56}. 
	
	The argument which was carried out in Step 1 of the proof of \thref{A37}, yields that $\Phi_{F,H}$ is of zero 
	type with respect to the order $\frac 12$. Let $C>0$ be such that $|F(z)|,|F^\#(z)|\leq C\mf m_{E_1}|_D(z)$, 
	$z\in D$. Moreover, let $z_0$ be a zero of $E_1$. 
	The basic estimate \eqref{A47}, used with $G(z):= (z-z_0)^{-1}E_1(z)$ and $w\in D$, gives 
	\begin{multline*}
		|\Phi_{F,H}(w)| \leq \Big(\int_{\bb R}\Big|\frac{F(t)}{E(t)}\Big|^2\frac{dt}{|t-w|^2}\Big)^{1/2}\cdot\|H\|_{\mc H}+ 
			\\[1mm]
		+C\frac{|w-z_0|}{|w+i|}\Big(\int_{\bb R}\Big|\frac{G(t)}{E(t)}\Big|^2\frac{dt}{|t-w|^2}\Big)^{1/2}
		\cdot\|H\|_{\mc H},\qquad w\in D
		\,.
	\end{multline*}
	However, since $F\in\mc H$ and $G\in\mc L\subseteq\mc H$, we have $F,G\in L^2(|E(t)|^{-2}\,dt)$. Moreover, 
	for $w\in D$, $|t-w|\geq y_0>0$. Hence, we may apply the Bounded Convergence Theorem to obtain 
	\[
		\lim_{\substack{|w|\to\infty\\ w\in D}}|\Phi_{F,H}(w)|=0
		\,.
	\]
	Since $\Phi_{F,H}$ is of order $\frac 12$ and zero type, the Phragm\'{e}n--Lindel{\"o}f Principle implies that 
	$\Phi_{F,H}$ vanishes identically. 
	
	Since $F\in R_{\mf m_{E_1}|_D}(\mc H)$ and $H\in\mc H\ominus\mc L$ were arbitrary, we conclude with the help 
	of \leref{A17} that $R_{\mf m_{E_1}|_D}(\mc H)\subseteq\breve{\mc L}$. 
	Thus, $\mc L\subseteq \mc R_{\mf m_{E_1}|_D}(\mc H)\subseteq\breve{\mc L}$. 
	
	In order to complete the proof, 
	assume on the contrary that $\mc L\subsetneq\breve{\mc L}$ and $\mc R_{\mf m_{E_1}|_D}(\mc H)=\breve{\mc L}$. Let 
	$\alpha\in[0,\pi)$ be such that 
	\[
		\breve{\mc L}=\mc L\oplus\spn\big\{e^{i\alpha}E_1-e^{-i\alpha}E_1^\#\big\}
		\,.
	\]
	Note that, in particular, $e^{i\alpha}E_1-e^{-i\alpha}E_1^\#\not\in\mc L$. 
	Since $R_{\mf m_{E_1}|_D}(\mc H)\nsubseteq\mc L$, 
	we can find a function $F\in\mc H(E_1)$ and a constant $\lambda\in\bb C\setminus\{0\}$, such that 
	\[
		F+\lambda\big(e^{i\alpha}E_1-e^{-i\alpha}E_1^\#\big)\in R_{\mf m_{E_1}|_D}(\mc H)
		\,.
	\]
	Set $\Theta:=e^{-2i\alpha}E_1^{-1}E_1^\#$ and consider the associated
        model subspace $\mc K_\Theta := H^2\ominus\Theta H^2$ of the Hardy space. 
	Recall that the mapping $F\mapsto F/E_1$ is a unitary transform of 
	$\mc H(E_1)$ onto $\mc K_\Theta$. Then 
	\[
		\bigg|\frac 1\lambda e^{-i\alpha}\underbrace{E_1^{-1}F}_{\in\mc K_\Theta}+1-\Theta\bigg|
                \lesssim\frac{1}{|z+i|}, \qquad z\in D
		\,.
	\]
	\thref{A60} implies that $1-\Theta\in\mc K_\Theta$. This contradicts the fact that 
	$e^{i\alpha}E_1-e^{-i\alpha}E_1^\#\not\in\mc L$. 

	\hspace*{0pt}\\[-2mm]\textit{Step 2: The case $y_0=0$.} We show that majorization remains present on each 
	fixed ray $iy+[h,\infty)$, $y>0$. This reduces the case $y_0=0$ to the case already settled in Step 1. 
	
	Let $F\in R_{\mf m_{E_1}|_D}(\mc H)$ be given. 
	Consider the function $f(z):= E_1^{-1}(z)\cdot zF(z)$. Then $f$ is of bounded type in $\bb C^+$. 
	Since $F$ and $E_1$ are entire functions of order $\frac 12$, we certainly have $\mt f=0$. Moreover, $f$ 
	has an analytic continuation to some domain which 
	contains the closure of $\bb C^+$. We conclude that 
	$f$ belongs to the Smirnov class $\mc N_+$. Hence $\log|f|$ is majorized throughout the half-plane $\bb C^+$ by 
	the Poisson integral of its boundary values: 
	\[
		\log|f(z)|\leq\frac{y}{\pi}\int_{\bb R}\frac{\log|f(t)|}{(t-x)^2+y^2}\,dt
		,\qquad z=x+iy\in\bb C^+
		\,.
	\]
	Since $F\in R_{\mf m_{E_1}|_D}(\mc H)$, the function $f$ is bounded on $D=[h,\infty)$. Again, since 
	$F$ and $E_1$ are of order $\frac 12$, we have 
	\[
		\int_{\bb R}\frac{\big|\log|tF(t)|\big|}{t^2+1}\,dt<\infty,\qquad 
		\int_{\bb R}\frac{\big|\log|E_1(t)|\big|}{t^2+1}\,dt<\infty
		\,,
	\]
	cf.\ \cite[p. 50, Theorem]{koosis:1988}. Hence we may estimate 
	\[
	\begin{aligned}
		\log|f(z)| 
		& \leq\frac{y}{\pi}\int_{[h,\infty)}\frac{\log^+|f(t)|}{(t-x)^2+y^2}\,dt+
		\frac{y}{\pi}\int_{(-\infty,h)}\frac{\log^+|tF(t)|}{(t-x)^2+y^2}\,dt \\
		& +\frac{y}{\pi}\int_{(-\infty,h)}\frac{\log^-|E_1(t)|}{(t-x)^2+y^2}\,dt
		,\qquad z=x+iy\in\bb C^+
		\,.
	\end{aligned}
	\]
	Since $f$ is bounded on $[h,\infty)$, the first summand is bounded independently of $z\in\bb C^+$. The second and 
	third summands are, for each fixed $y>0$, nonincreasing and nonnegative functions of $x\geq h$. In particular, they are 
	bounded on each ray $iy+[h,\infty)$, $y>0$. It follows that, for each fixed positive value of $y$, we have 
	$|F(z)|\lesssim\mf m_{E_1}(z)$, $z\in iy+[h,\infty)$. 
\end{proof}

The following two examples show that the statement in \thref{A48} is in some ways sharp. 

\begin{example}{A38}
	There exists a space $\mc H=\mc H(E)$ with $E$ of order $\frac 12$ and finite type, 
	and a subspace $\mc L \in \Sub^*(\mc H)$, such that $\mc R_{\mf m_{E_1}|_D}(\mc H) \neq \mc L$.
	To show this we construct a  matrix $W$ with components of order $\frac 12$
	such that the space $\mc K(W)$ exists and such that
	one of its rows is bounded on $(0,\infty)$.

	Consider the two auxiliary functions 
	\[
		G(z):= \prod\limits_{n\in\mathbb{N}}\left(1-\frac{z}{n^2-i}\right)
		= c\frac{\sin (\pi\sqrt{z+i})}{\pi\sqrt{z+i}}, \qquad 
                c= \prod\limits_{n\in\mathbb{N}} \frac{n^2-i}{n^2},
	\]
	\[
		\tilde G(z) :=  \prod\limits_{n\in\bb N} \bigg(1-\frac{z}{n^2-in}\bigg).
	\] 
	Let $ x\in (k-1/2,k+1/2)$, $k\in\mathbb{N}$. Then we can write
	\[
		|G(x^2)|= \left| \frac{k^2-x^2+i}{k^2-i} \right|
		\prod\limits_{n\ne k} \left| 1-\frac{x^2}{n^2-i} \right|,
	\]
	and it is easy to see that 
	\[	       
		|G(x^2)| \asymp\left| \frac{k^2-x^2+i}{k^2-i} \right|
		\cdot \frac{| \sin\pi x|}{x}\cdot \left|1-\frac{x^2}{k^2}\right|^{-1}
		\asymp \left| \frac{k^2-x^2+i}{x(x+k)} \right|
		\cdot \frac{| \sin\pi (x-k)|}{|x-k|}.
	\]
	Next, note that
	\[
		\bigg|\frac {\tilde G(x^2)}{G(x^2)}\bigg|^2 
		 = \prod\limits_{n\in\bb N}
		\frac{(x^2-n^2)^2+n^2}{(x^2-n^2)^2+1}\cdot\frac{n^4+1}{n^4+n^2}    
		 \asymp \frac{(x^2-k^2)^2 + k^2}{(x^2-k^2)^2+1}.
	\]
	Combining this, we conclude that
	\[
		|\tilde G(x^2)|\asymp \frac{|x^2-k^2+ki|}{x(x+k)}\asymp k^{-1}\asymp x^{-1}, 
		\qquad x>1.
	\]
	Set $E_0(z) = (z+i)(\tilde G(z))^2$. Then $E_0$ is of order $1/2$ and finite type,
	we have $|E_0(x)|\asymp 1$, $x>1$, and $\log|E(x)| \asymp |x|^{1/2}$, 
	$x\to -\infty$. Therefore $1\in \Assoc(\mc H(E_0))$. Let $E_0 = A_0-iB_0$.
	Changing slightly the function $E_0$ we may assume that $A_0(0)=1$, $B_0(0)=0$.
	Then by \cite[Theorems 27, 28]{debranges:1968},
	there exist real entire functions $C_0, D_0$, 
	with $D_0(0)=1$, $C_0(0)=0$, such that for the matrix
	\[
		W:= \begin{pmatrix}
		A_0 & B_0\\
		C_0 & D_0
		\end{pmatrix}
	\] 
	the space $\mc K(W)$ exists. Thus also the space $\mc K(\tilde W)$ exists, where 
	\[
		\tilde W:= 
		\begin{pmatrix}
			D_0 & B_0\\
			C_0 & A_0
		\end{pmatrix}
		\,.
	\] 
	Let $\mc H(E_1)$ be an arbitrary de Branges space, and set $E=A-iB$, where $(A,B):= (A_1,B_1)\tilde W$. 
	Define 
	\[
		\begin{pmatrix}
		f_+ \\
		f_- \end{pmatrix} 
		:= 
		\frac{\tilde W(z)J - J}{z}
		\begin{pmatrix}
		1 \\
		0 \end{pmatrix} =
		\begin{pmatrix}
		\frac{B_0(z)}{z}\\
		\frac{A_0(z)-1}{z}\end{pmatrix},  \qquad     
		J= \begin{pmatrix} 0 & -1\\ 1 & 0 \end{pmatrix}\,.
	\]
	Then $\binom{f_+}{f_-} \in \mc K(W) $and so 
        $f_+ A_1 + f_- B_1 \in \mc H(E)\ominus \mc L$. However, this function 
	is majorized by $\mf m_{E_1}$ on $(0,\infty)$. Thus $\mc R_{\mf m_{E_1}|_{(0,\infty)}}(\mc H(E))\supsetneq\mc H(E_1)$. 
\end{example}

\begin{example}{A41}
	There exist spaces $\mc H=\mc H(E)$, $\mc L = \mc H(E_1)$, with functions $E$, $E_1$, 
	of arbitrarily small order such that $\mc L\in\Sub\mc H$ and, for $D=\bb R+iy_0$, 
	\[
		\mc R_{\nabla_{\!\mc L}|_D}(\mc H) \neq \mc L.
	\]
	Thus, in the statement of \thref{A48} we cannot replace $\mc R_{\mf m_{E_1}|_D}(\mc H)$ 
	by $\mc R_{\nabla_{\!\mc L}|_D} (\mc H)$. 

	To show this, let $\alpha>1$, let $t_n= |n|^{\alpha}$, $n\in \bb Z \setminus \{0\}$, 
	and let $\mu_n = |n|^{2\alpha-2}$. Put
	\[
		q(z) = \sum_{n\in \bb Z\setminus\{0\}} 
		\mu_n \bigg(\frac{1}{t_n-z}- \frac{1}{t_n}\bigg).
	\]
	The series converges since
	\[
		\sum_n \frac{\mu_n}{t_n^2} = \sum_n \frac{|n|^{2\alpha-2}}{|n|^{2\alpha}}
		= \sum_n \frac{1}{n^2}<\infty.
	\]
	There exist real entire functions $A_1$ and $B_1$ such that $q=B_1/A_1$ 
	and $\mc H(E_1)$ exists.
	We show that for $\mc L=\mc H(E_1)$ and $y_0>0$
	we have
	\begin{equation}\label{A40}
		|A_1(x+iy_0)| \lesssim \nabla_{\!\mc L}(x+iy_0), \qquad x\in \bb R.
	\end{equation}
	Put $\Theta = E_1^{-1} E_1^\#$. Then \eqref{A40} is equivalent to
	\begin{equation}\label{A44}
		|1+\Theta(x+iy_0)|^2 \lesssim 1-|\Theta(x+iy_0)|^2,
		\qquad x\in \bb R.
	\end{equation}
	Also, $q=i\frac{1-\Theta}{1+\Theta}$ and so
	\[
		\Im q(x+iy_0) = \frac{1-|\Theta(x+iy_0)|^2}{|1+\Theta(x+iy_0)|^2}
		=y_0\sum_n \frac{|n|^{2\alpha-2}}{(x-|n|^\alpha)^2+y_0^2}.
	\]
	It is easy to see that $\Im q(x+iy_0) \gtrsim 1$, which implies \eqref{A44}.
	Indeed, let $x\in [k^\alpha, (k+1)^\alpha]$, $k\in \bb N$.
	Then $|x-k^\alpha|\lesssim k^{\alpha-1}$ with constants independent of $k$.
	Hence
	\[
		\sum_n \frac{|n|^{2\alpha-2}}{(x-|n|^\alpha)^2+y_0^2}>
		\frac{k^{2\alpha-2}}{(x-k^\alpha)^2+y_0^2}\gtrsim 1.
	\]
\end{example}

\begin{remark}{A53}
	We return to the comment made in \reref{A39}. Seeking an example that the growth assumption in \thref{A37} is necessary, 
	the first idea would be to proceed in the same way as in \exref{A38}. But this is not possible. The reason is that 
	there exists no function $E_0\in\HB$, such that $1\in\Assoc (\mc H(E_0))$ and
	$E_0 = A_0-iB_0$ is bounded on $D$. Indeed, if $E_0$ would have these properties, then 
	$\frac{1}{(z+i)E_0} \in H^2$, 
	which implies that $|E_0(x+i)|\gtrsim |x+i|^{-1}$, $x\in \bb R$. 
	Also since $E_0$ is not a polynomial, $\log |E_0(iy)|> N \log y$,
	$y\to\infty$, for any fixed $N>0$. 
	Applying the Poisson formula in the angle $\{\Re z>0, \Im z>1\}$
	to $\log|E_0|$ ("Two Constant Theorem") we conclude that $E_0$ 
	is unbounded on $D$.
\end{remark}

\appendix
\makeatletter
\renewcommand{\@seccntformat}[1]{\@nameuse{the#1}.\quad}
\makeatother
\appendixsection{Estimates of inner functions on horizontal rays}

In this appendix we prove a theorem about 
asymptotic behavior of inner functions along horizontal rays.
It is well known that for any ray 
$D := e^{i\pi\beta}[0,\infty)$ with $0<\beta<1$, the estimate 
\[
	|e^{2i\alpha}-\Theta(z)|\lesssim |z+i|^{-1}, \qquad z\in D,
\]
is equivalent to $e^{2i\alpha} -\Theta \in\mc H(E)$. 
If $\Theta= E^{-1} E^\#$ is a meromorphic inner function
the latter condition means that $e^{i\alpha}E -e^{-i\alpha}E^\# 
\in\mc H(E)$. For the de Branges space setting see 
\cite[Theorem 22]{debranges:1968}, the case of general inner functions
is discussed, e.g, in \cite{baranov:2001}.
We show that an analogous and even stronger statement 
is true for the rays $iy_0 + [0,\infty)$, $y_0 > 0$. 

Each inner function $\Theta$ generates a 
model subspaces $\mc K_{\Theta}=H^2\ominus\Theta H^2$ of $H^2$. 

\begin{theorem}{A60}
	Let $\Theta$ be an inner function in $\bb C^+$ and let 
	$y_0>0$. Assume that there exists a function 
	$f\in\mc K_\Theta$ and a positive constant $C$, such that 
	\begin{equation}\label{A61}
		\big|f(x+iy_0)+1-\Theta(x+iy_0)\big|\leq\frac C{|x+iy_0|},\quad x>0
		\,.
	\end{equation}
	Then $1-\Theta\in\mc K_\Theta$. 
\end{theorem}

For the proof of this result we will combine weak type estimates for 
the Hilbert transform and properties of the Clark measures. 
We will throughout this appendix keep the following notation: 

\begin{enumerate}[$(i)$]
\item The Lebesgue measure on $\bb R$ is denoted by $m$. Moreover, $\Pi$ denotes the Poisson measure on $\bb R$, that is 
	$d\Pi(t)=(1+t^2)^{-1}dm(t)$. 
\item The set of all functions $q$ which are defined and analytic in $\bb C^+$ and have nonnegative real part 
	throughout this half-plane is denoted by $\mc C$. 
\item Recall that a function $q$ belongs to $\mc C$ if and only if it has an integral 
	representation of the form 
	\begin{equation}\label{A62}
		q(z)=-ipz+i\Im q(i)+\frac i\pi\int_{\bb R}\Big(\frac 1{z-t}+\frac t{1+t^2}\Big)d\mu(t)
		\,,
	\end{equation}
	where 
	\[
		p\geq 0,\qquad \mu\ \text{is a positive Borel measure},\ \  
	        \int_{\bb R}\frac{d\mu(t)}{1+t^2}<\infty
		\,.
	\]
	The data $p$ and $\mu$ in this representation are uniquely determined by 
	the function $q$ (see, e.g., \cite[5.3,5.4]{rosenblum.rovnyak:1994}).  
	Note that, if the function $q$ has a continuous extension to the closed half-plane $\bb C^+\cup\bb R$, 
	then the measure $\mu$ is absolutely continuous with respect to $m$ and 
	\[
		d\mu(t)=\Re q(t)\,dm(t)
		\,.
	\]
\item Two subclasses of $\mc C$ are defined as 
	\[
		\mc C_1:=\big\{q\in\mc C:\, p= \lim_{y\to+\infty}\frac 1y\Re q(iy)=0\big\}
		\,,
	\]
	\[
		\mc C_0:=\big\{q\in\mc C:\ \text{the limit}\, 
                \lim_{y\to+\infty}yq(iy)\,\ \text{exists} \big\}
		\,.
	\]
\end{enumerate}
Recall that $q\in\mc C_0$ if and only if in \eqref{A62} we have 
\[
	p=0,\quad \int_{\bb R}d\mu(t)<\infty,\quad 
        \Im q(i)=-\frac 1\pi\int_{\bb R}\frac t{1+t^2}\,d\mu(t)
	\,,
\]
i.e.,\ if and only if $q$ can be represented in the form 
\[
	q(z)=\frac i\pi\int_{\bb R}\frac{d\mu(t)}{z-t}
\]
with a finite positive Borel measure $\mu$, see e.g.\ \cite[Theorem 6.4]{gorbachuk.gorbachuk:1997}. In this case we have 
\[
	\int_{\bb R}d\mu(t)=\pi\lim_{y\to\infty}yq(iy)
	\,.
\]
Weak type estimates enter the discussion in the form \leref{A63} below, and can 
be used to conclude that $1-\Theta\in\mc K_\Theta$, cf.\ \leref{A64}. 

\begin{lemma}{A63}
	Let $y_0>0$ be given. 
	\begin{enumerate}[$(i)$]
	\item Whenever $q\in\mc C_1$, we have 
		\[
			\lim_{a\to+\infty} a\cdot\Pi\Big(\big\{x\in\bb R:\,|q(x+iy_0)|>a\big\}\Big)=0
			\,.
		\]
	\item There exists a positive constant $A$, such that 
		\[
			a\cdot m\Big(\big\{x\in\bb R:\,|q(x+iy_0)|>a\big\}\Big)\leq A\cdot\lim_{y\to+\infty}yq(iy)
			,\quad a>0,\ q\in\mc C_0
			\,.
		\]
	\end{enumerate}
\end{lemma}
\begin{proof}
	Let $q\in\mc C$, and consider the function 
	\[
		Q(z):=q(z+iy_0),\qquad z\in\bb C^+\cup\bb R
		\,.
	\]
	Then $Q$ is continuous in $\bb C^+\cup\bb R$ and belongs to $\mc C$. 
	In particular, 
	\[
		\Re Q(x) = py_0 + \frac{y_0}{\pi} \int_{\bb R} \frac{d\mu(t)}{(t-x)^2+y_0^2},
	\]
	and it is easy to see that $\Re Q(x)\in L^1(\Pi)$. 
	Moreover, $Q\in\mc C_1$ (or $Q\in\mc C_0$) if and only if $q$ has the respective property. 	

	For the proof of $(i)$, assume that $q\in\mc C_1$. Then we have 
	\[
	\begin{aligned}
		\Im Q(x)-\Im Q(i) & =\lim_{y\searrow 0}\Im Q(x+iy)-\Im Q(i) \\
		& =\lim_{y\searrow 0}\frac 1\pi\int_{\bb R}\Big(\frac{x-t}{(x-t)^2+y^2}+\frac t{1+t^2}\Big)\Re Q(t)\,dt
		\,.
	\end{aligned}
	\]
	By Kolmogorov's Theorem on the harmonic conjugate, to be more specific by \cite[p. 65, Corollary]{koosis:1988}, 
	we have 
	\[
		\lim_{a\to+\infty}a\cdot\Pi\Big(\big\{x\in\bb R:\,|\Im Q(x)-\Im Q(i)|>a\big\}\Big)=0
		\,.
	\]
	Since also 
	\[
		a\cdot\Pi\Big(\big\{x\in\bb R:\,|\Re Q(x)|>a\big\}\Big)\leq\int_{\bb R}\chi_{\{\Re Q>a\}}\Re Q\,d\Pi
		\stackrel{a\to+\infty}{\longrightarrow} 0
		\,,
	\]
	the desired limit relation follows. 
	
	For the proof of $(ii)$, assume that $q\in\mc C_0$. Then we have $\Re Q(x)\in L^1(m)$, and 
	\[
		Q(z)=\frac i\pi\int_{\bb R}\frac{\Re Q(t)}{z-t}\,dm(t),\qquad z\in\bb C^+
		\,.
	\]
	This shows that $\Im Q(x)$ is the standard Hilbert transform of $\Re Q(x)$. Thus, by 
	\cite[V, Lemma 2.8]{stein.weiss:1971}, we have the weak type estimate 
	\[
		m\Big(\big\{x\in\bb R:\,|\Im Q(x)|>a\big\}\Big)\leq\frac ea\int_{\bb R}\Re Q(t)\,dm(t)=
		\frac ea \pi\lim_{y\to+\infty}yq(iy)
		\,,
	\]
	where $e$ is the Euler number. Since 
	\[
		m\Big(\big\{x\in\bb R:\,|\Re Q(x)|>a\big\}\Big)\leq\frac 1a\int_{\bb R}\chi_{\{\Re Q>a\}}\Re Q\,dm\leq
		\frac{\pi}a\lim_{y\to+\infty}yq(iy)
		\,,
	\]
	we obtain the desired estimate, e.g.\ with the constant $A:=\pi \sqrt 2(1+e)$. 
\end{proof}

\begin{lemma}{A64}
	Let $\Theta$ be an inner function in $\bb C^+$, and let $y_0>0$. Assume that there exist 
	positive constants $c,c'$, and $r_0$, such that 
	\begin{equation}\label{A65}
		m\Big(\Big\{x\in[r,2r]:\,|1-\Theta(x+iy_0)|\leq\frac c{|x+iy_0|}\Big\}\Big)\geq c'r,\qquad r\geq r_0
		\,.
	\end{equation}
	Then $1-\Theta\in\mc K_\Theta$. 
\end{lemma}
\begin{proof}
	Consider the function 
	\[
		q(z):=\frac{1+\Theta(z)}{1-\Theta(z)},\qquad z\in\bb C^+
		\,.
	\]
	For $r>0$ set 
	\[
		M_r:=\Big\{x\in[r,2r]:\,|1-\Theta(x+iy_0)|
		\leq\frac c{|x+iy_0|}\Big\}
		\,.
	\]
	Then, by our hypothesis \eqref{A65}, we have 
	$m(M_r)\geq c'r$, $r\geq r_0$. 
	Assume that $a>1$ and let $x\in M_r$ with $r>ca$. Then 
	\[
	\begin{aligned}
		|q(x+iy_0)| &=\Big|\frac{1+\Theta(x+iy_0)}{1-\Theta(x+iy_0)}\Big|
		\geq\frac{2-|1\!-\!\Theta(x+iy_0)|}{|1\!-\!\Theta(x+iy_0)|}\geq \\
		& \geq \frac {2 |x+iy_0|}{c} -1 \ge 2a-1> a	
		\,,
	\end{aligned}
	\]
	since $x \geq r \geq ca$. Thus, we have
	\[
		M_r \subseteq\big\{x\in\bb R:\,|q(x+iy_0)|>a\big\},\qquad 
                a> 1, \ r\ge ca
		\,.
	\]
	It follows that for $a> 1$, $r\ge ca$, 
	\[
		\Pi\Big(\big\{x\in\bb R:\,|q(x+iy_0)|>a\big\}\Big)\geq\Pi(M_r)=
		\int_{M_r}\frac{dm(t)}{1+t^2}\geq
	\]
	\begin{equation}\label{A66}
		\geq\frac 1{1+4r^2}m(M_r)\geq\frac{c'r}{1+4r^2}
		\,.
	\end{equation}
        If $a \geq\frac{r_0}{c}$, then $r:=ca \geq r_0$, and we may 
        use this particular value of $r$ in \eqref{A66}. It follows that,
	for $a>\max (1, r_0/c)$, 
	\[
		\Pi\Big(\big\{x\in\bb R:\,|q(x+iy_0)|>a\big\}\Big)\geq
                \frac{c'c a}{1+4c^2 a^2} \geq \frac{d}{a}
		\,,
	\]
	where the constant $d$ depends only on $c$ and $c'$.

	The function $q$ is analytic in $\bb C^+$ and has nonnegative real part throughout this half-plane. 
	By \leref{A63}, $(i)$, it cannot belong to the subclass $\mc C_1$, i.e.,\ we have 
	\[
		\lim_{y\to+\infty}\frac 1yq(iy)>0
		\,.
	\]
	However, this property of $q$ is, e.g.\ by the discussion in \cite{baranov:2001}, equivalent to $1-\Theta$ belonging 
	to $\mc K_\Theta$. 
\end{proof}

\begin{proofof}{\thref{A60}}
	Assume on the contrary that $1-\Theta\not\in\mc K_\Theta$. Our aim is to show that, under the 
	assumptions of the theorem, the relation \eqref{A65} holds for some appropriate values of $c,c'$ and $r_0$. 
	Once this has been achieved, \leref{A64} implies that $1-\Theta\in\mc K_\Theta$, and we 
	have derived a contradiction. 

	The function $1-\Theta$ does not belong to $\mc K_\Theta$
	if and only if $q= \frac{1+\Theta}{1-\Theta}$ is in $\mc C_1$, that is, 
	$p=0$ in \eqref{A62}. Thus,
	\[
		\frac{1+\Theta(z)}{1-\Theta(z)} =i\Im q(i)+
                \frac i\pi\int_{\bb R}\Big(\frac 1{z-t}+\frac t{1+t^2}\Big)d\mu(t)
		\,,
	\]
        The measure $\mu$, called the Clark measure, has many 
        important properties (see, e.g., 
        \cite[Vol. 2, Part D, Chapter 4]{nikolski:2002}).
	In particular, it was shown in \cite{poltoratski:1993}
	that each function $f\in \mc K_\Theta$ has radial boundary values
	$\mu$-a.e. and the restriction operator
	$f\mapsto f|_{{\rm supp}(\mu)}$ is a unitary operator
	from $\mc K_\Theta$ onto $L^2(\mu)$. Note that
	$\Theta =1$ $\mu$-a.e. on ${\rm supp}(\mu)$.

	For $z\in \bb C^+$ denote by $k_z$ the reproducing kernel of $\mc K_\Theta$,
	\[
		k_z(\zeta) = \frac{i}{2\pi}\cdot 
		\frac{1-\overline{\Theta(z)}\Theta(\zeta)}{\zeta-\overline z} \,.
	\]
	Then, for $f\in \mc K_\Theta$ and $z\in \bb C^+$, we have
	\[
		f(z) = (f, k_z)_{L^2(\mu)} = \frac{1-\Theta(z)}{2\pi i} 
		\int_{\bb R} \frac{f(t)}{t-z}\, d\mu(t) \,,
	\]
	since $\Theta = 1$ $\mu$-a.e. 

	Now let $f\in K_\Theta$ be a function as in the hypothesis
	of \thref{A60}. Note that, for each $M>0$, there exists 
        a positive constant $C_M$ such that 
	\begin{equation}\label{A68}
		\bigg|  \int_{[-M, M]} \frac{f(t)}{t-x-iy_0}\, d\mu(t) \bigg|
                \leq\frac{C_M}{|x+iy_0|},\qquad x\in\bb R
		\,.
	\end{equation}
	
	Let $\epsilon>0$ be fixed. We have $f\in L^2(\mu)$
	and $(|t|+1)^{-1} \in L^2(\mu)$ and so
	$(|t|+1)^{-1} f \in L^1(\mu)$.
	Using \eqref{A68}, we obtain that there exists $M_\epsilon>0$ and $C_\epsilon>0$ such that the function 
	($\mu_\epsilon := \frac{1}{2\pi} \mu|_{\bb R \setminus [-M_\epsilon, M_\epsilon]}$) 
	\begin{equation}\label{A69}
		f_\epsilon(z) := 
		\frac{1-\Theta(z)}{i} 
		\int_{\bb R} \frac{f(t)}{t-z}\, d\mu_\epsilon(t) \,,
	\end{equation}
	satisfies 
	\begin{enumerate}[(1)]
	\item $\int_{\bb R}\frac{|f(t)|}{|t|}\,d\mu_\epsilon(t) <\epsilon$, 
	\item $|f_\epsilon(x+iy_0)+1-\Theta(x+iy_0)|\leq\frac{C_\epsilon}{|x+iy_0|}$, $x>0$.
	\end{enumerate}
	For the time being, let $\epsilon>0$ be arbitrary; we will make a particular 
        choice later. 

	The representation \eqref{A69} of the function $f_\epsilon$ may be rewritten as 
	\begin{equation}\label{A72}
		f_\epsilon(z)=
		\frac{1-\Theta(z)}{i} 
		\bigg(\underbrace{\int_{\bb R} \frac{f(t)}{t}\, d\mu_\epsilon(t)
		+ z \int_{\bb R} \frac{f(t)}{t}\cdot 
		\frac{d\mu_\epsilon(t)}{t-z}}_{=:\gamma_\epsilon(z)}\bigg)
		\,.
	\end{equation}
	Let $u(t) = \Re \frac{f(t)}{t}$ and let $u_+ = \max\{u,0\}$, $u_-=u_+ - u$. Set 
	\begin{equation}\label{A70}
		q(z) := \frac 1i \int_{\bb R} \frac{u_+(t)}{t-z}\,d\mu_\epsilon(t)
		\,.
	\end{equation}
	Then $q\in\mc C_0$ and 
	\[
		\lim_{y\to+\infty}yq(iy)= \int u_+(t)\, d\mu_\epsilon(t)
                \leq \epsilon
		\,.
	\]
	Using \leref{A63}, $(ii)$, we obtain 
	\[
		m\Big(\big\{x\in\bb R:\,|q(x+iy_0)|>a\big\}\Big)\leq \frac Aa 
                \epsilon,\qquad a>0
		\,.
	\]
	The same argument applies when we take $u_-$, as well as
        $v_+$ and $v_-$ for $v(t) = \Im \frac{f(t)}{t}$, instead of $u_+$ 
        in the definition \eqref{A70} of $q$. Altogether, we conclude that 
	\[
		m\Big(\big\{x\in\bb R:\,
                \bigg|\int_{\bb R} \frac{f(t)}{t}\cdot 
		\frac{d\mu_\epsilon(t)}{t-z}\bigg|>a\big\}\Big)\leq \frac{16A}{a}
		\epsilon,\quad a>0
		\,.
	\]
	Let $r>0$ and $x\in[r,2r]$ be given, then 
	\[
		|\gamma_{\epsilon}(x+iy_0)|\leq
		\underbrace{\int_{\bb R} \frac{|f(t)|}{|t|}\, 
		d\mu_\epsilon(t)}_{\leq \epsilon}
		+ (2r+y_0) \bigg| \int_{\bb R} \frac{f(t)}{t}\cdot 
		\frac{d\mu_\epsilon(t)}{t-z}\bigg|
		\,.
	\]
	It follows that, for any $r>0$ and $a>0$,  
	\begin{multline*}
		m\Big(\big\{x\in[r,2r]:\,|\gamma_{\epsilon}(x+iy_0)|>\epsilon+
                (2r+y_0)a\big\}\Big)\leq\\
		\leq m\bigg(\Big\{x\in [r,2r]:\,\Big|
                \int_{\bb R} \frac{f(t)}{t}\cdot \frac{d\mu_\epsilon(t)}{t-z}\Big|
		>a\Big\}\bigg) \leq\frac {16A}{a}\epsilon
			\,,
	\end{multline*}
	and hence 
	\[
		m\Big(\big\{x\in[r,2r]:\,|\gamma_{\epsilon}(x+iy_0)|\leq 
                \epsilon+(2r+y_0) a \big\}\Big)\geq
		r-\frac{16A}{a}\epsilon
	\]
	Assume that $r\geq y_0$, and use this inequality for the 
        particular value $a:=\frac{\sqrt\epsilon}r$ of $a$. Then 
	it follows that 
	\begin{equation}\label{A71}
		m\Big(\big\{x\in[r,2r]:\,|\gamma_{\epsilon}(x+iy_0)|
                \leq \epsilon+3\sqrt\epsilon\big\}\Big)\geq
		r(1-16A\sqrt\epsilon)
		\,.
	\end{equation}
	At this point we make a particular choice of $\epsilon$, namely, 
        we take $\epsilon>0$ so small that 
	$\epsilon+3\sqrt\epsilon\leq\frac 12$ and $16A\sqrt\epsilon
        \leq\frac 12$. Then \eqref{A71} gives 
	\[
		m\Big(\Big\{x\in[r,2r]:\,|\gamma_{\epsilon}(x+iy_0)|\leq \frac 12\Big\}\Big)\geq\frac 12r,\qquad r\geq y_0
		\,.
	\]
	However, if $x\in[r,2r]$ is such that 
        $|\gamma_{\epsilon}(x+iy_0)|\leq \frac 12$, then by 
        the hypothesis \eqref{A61} and the relation \eqref{A72} we obtain that 
	\[
		|1-\Theta(x+iy_0)|=\frac{|f_\epsilon(x+iy_0)+1-\Theta(x+iy_0)|}
                {|1-i\gamma_\epsilon(x+iy_0)|}
		\leq\frac{2C_\epsilon}{|x+iy_0|}
		\,.
	\]
	We conclude that, for $r\ge y_0$, 
	\begin{multline*}
		m\Big(\Big\{x\in[r,2r]:\,|1-\Theta(x+iy_0)|\leq\frac{2C}{|x+iy_0|}\Big\}\Big)\geq \\
		\geq m\Big(\Big\{x\in[r,2r]:\,|\gamma_{\epsilon}(x+iy_0)|\leq \frac 12\Big\}\Big)
			\geq\frac 12r
			\,,
	\end{multline*}
	i.e.\ \eqref{A65} holds. 
\end{proofof}

\appendixsection{Summary of results}

Let $\mc H=\mc H(E)$ be a de~Branges space and let $\mc L=\mc H(E_1)\in\Sub\mc H$. 

%
%
\begin{flushleft}
	\textbf{a. Necessary conditions for $\mc L=\mc R_{\mf m}(\mc H)$.}
\end{flushleft}
\vspace*{0mm}
\[
	\begin{array}{l|l}
		D & \parbox{25mm}{condition on $\mc L$} \\ \hline\hline
		w\in\bb R\setminus\qu D \rule{0pt}{5mm} & \mf d_{\mc L}(w)=\mf d_{\mc H}(w)\\[2mm] \hline
		\parbox[t]{60mm}{$\Im z\leq \psi(\Re z)$, $z\in D$,\\[1mm]
			$\psi$ positive, even, increasing on $[0,\infty)$,\\[1mm] 
			$\int_0^\infty(t^2+1)^{-1}\psi(t)\,dt<\infty$} \rule{0pt}{5mm} & 
			\mt_{\mc H}\mc L=0 \\[13mm] \hline
	\end{array}
\]

%
%
\newpage
\begin{flushleft}
	\textbf{b. Sufficient conditions for $\mc L=\mc R_{\mf m}(\mc H)$.}
\end{flushleft}
\vspace*{0mm}
\[
	\begin{array}{l|l|c|c}
		D & \parbox{25mm}{representation of $\mc L$} & \parbox{20mm}{assumption on $\mc L$} & 
			\parbox{20mm}{assumption on $\mc H$} \\[3mm] \hline\hline
		\bb R\cup i[0,\infty) \rule{0pt}{5mm} & \mc L=\mc R_{\nabla_{\!\mc L}|_D}(\mc H) & 
			\raisebox{1mm}{\rule{10mm}{0.5pt}} & 
			\raisebox{1mm}{\rule{10mm}{0.5pt}} \\[1mm]
		\parbox[t]{28mm}{$\bb R\cup e^{i\pi\beta}[0,\infty)$\\ $\beta\in(0,1)$} \rule{0pt}{5mm} & 
			\mc L=\mc R_{\mf m_{E_1}|_D}(\mc H) & \raisebox{1mm}{\rule{10mm}{0.5pt}} & 
			\raisebox{1mm}{\rule{10mm}{0.5pt}} \\[6mm] \hline

		i[h,\infty),\ h>0 \rule{0pt}{5mm} & \mc L=\mc R_{\nabla_{\!\mc L}|_D}(\mc H) & \mf d_{\mc L}=\mf d_{\mc H} & 
			\raisebox{1mm}{\rule{10mm}{0.5pt}} \\[2mm]
		\parbox[t]{28mm}{$e^{i\pi\beta}[h,\infty),\ h>0,$\\ $\beta\in(0,\frac 12)\cup(\frac 12,1)$} \rule{0pt}{5mm} & 
			\mc L=\mc R_{\nabla_{\!\mc L}|_D}(\mc H) & 
			\mf d_{\mc L}=\mf d_{\mc H} & \parbox[t]{25mm}{order $(2-2\beta)^{-1}$ zero type} \\[6mm] \hline

		\bb R \rule{0pt}{5mm} & \mc L=\mc R_{\mf m_{E_1}|_D}(\mc H) & \mt_{\mc H}\mc L=0 & 
			\raisebox{1mm}{\rule{10mm}{0.5pt}}\\[2mm]
		& \mc L\!\subseteq\mc R_{\nabla_{\!\mc L}|_D}(\mc H)\subseteq\!\breve{\mc L} & 
			\parbox[t]{21mm}{$\mt_{\mc H}\mc L=0$,\\ $\sup_{\bb R}\varphi_{E_1}'\!<\!\infty$} & 
			\raisebox{1mm}{\rule{10mm}{0.5pt}}\\[7mm]
		& \mc L=\mc R_{\nabla_{\!\mc L}|_D}(\mc H) & \mf d_{\mc L}=\mf d_{\mc H} & 
			\parbox[t]{25mm}{$\sup_{\bb R}\varphi_E'\!<\!\infty$,\\ $\forall\mc K:\mc K=\breve{\mc K}$} \\[6mm] \hline

		\bb R+ih,\ h>0 \rule{0pt}{5mm} & \mc L=\mc R_{\mf m_{E_1}|_D}(\mc H) & 
			\parbox[t]{17mm}{$\mf d_{\mc L}=\mf d_{\mc H}$,\\ $\mt_{\mc H}\mc L=0$} & 
			\raisebox{0mm}{\rule{10mm}{0.5pt}} \\[6mm]
		& \mc L\!\subseteq\mc R_{\nabla_{\!\mc L}|_D}(\mc H)\subseteq\!\breve{\mc L} & 
			\parbox[t]{17mm}{$\mf d_{\mc L}=\mf d_{\mc H}$,\\ $\mt_{\mc H}\mc L=0$} & 
			\raisebox{0mm}{\rule{10mm}{0.5pt}} \\[6mm] \hline
			
		\parbox[t]{29mm}{$iy_0{\scriptstyle+}[h,\infty),\ h\in\bb R$,\\ $y_0\geq 0$} \rule{0pt}{5mm} & 
			\mc L=\mc R_{\mf m_{E_1}|_D}(\mc H) & 
			\mf d_{\mc L}=\mf d_{\mc H} & \parbox[t]{25mm}{order $\frac 12$\\ zero type} \\[6mm] \hline
	\end{array}
\]



{\footnotesize
\begin{flushleft}
	A.\,Baranov\\
	Department of Mathematics and Mechanics\\
	Saint Petersburg State University\\
	28, Universitetski pr.\\
	198504 Petrodvorets\\
	RUSSIA\\
	email: a.baranov@ev13934.spb.edu\\[5mm]
\end{flushleft}
\begin{flushleft}
	H.\,Woracek\\
	Institut f\"ur Analysis und Scientific Computing\\
	Technische Universit\"at Wien\\
	Wiedner Hauptstr.\ 8--10/101\\
	A--1040 Wien\\
	AUSTRIA\\
	email: harald.woracek@tuwien.ac.at\\[5mm]
\end{flushleft}
}

\end{document}